%% file: HaSkTr_arXiv.tex
\documentclass{article}
\usepackage{amsmath,amscd,amstext,amssymb,amsfonts}
\usepackage{amsthm}
\usepackage{graphics}
\usepackage{latexsym,empheq,mathtools}
\usepackage{color,xcolor}
\usepackage{bookman}
\usepackage{bm}
\usepackage{cite}
\usepackage{enumerate}
\usepackage{amscd}
\usepackage{tikz-cd}
\numberwithin{equation}{section}

\theoremstyle{plain}
\newtheorem{theorem}{Theorem}[section]
\newtheorem{proposition}[theorem]{Proposition}
\newtheorem{definition}[theorem]{Definition}

\theoremstyle{definition}
\newtheorem{remark}[theorem]{Remark}

\begin{document}

\newcommand{\SBsp}{\ensuremath{X^s_pB}}
\newcommand{\SBxp}[1]{\ensuremath{X^{#1}_pB}}
\newcommand{\SBxx}[2]{\ensuremath{X^{#1}_{#2}B}}
\newcommand{\SHsp}{\ensuremath{X^s_pH}}
\newcommand{\SHxp}[1]{\ensuremath{X^{#1}_pH}}
\newcommand{\SHxx}[2]{\ensuremath{X^{#1}_{#2}H}}
\newcommand{\TBsp}{\ensuremath{Y^s_pB}}
\newcommand{\TBxp}[1]{\ensuremath{Y^{#1}_pB}}
\newcommand{\TBxx}[2]{\ensuremath{Y^{#1}_{#2}B}}
\newcommand{\THsp}{\ensuremath{Y^s_pH}}
\newcommand{\THxp}[1]{\ensuremath{Y^{#1}_pH}}
\newcommand{\THxx}[2]{\ensuremath{Y^{#1}_{#2}H}}
\newcommand{\SAxp}[1]{\ensuremath{X^{#1}_pA}}
\newcommand{\TAxp}[1]{\ensuremath{Y^{#1}_pA}}
\newcommand{\snp}{\ensuremath{x^n_p}}
\newcommand{\tnp}{\ensuremath{y^n_p}}
\newcommand{\dnp}{\ensuremath{d^n_p}}
\newcommand{\dnx}[1]{\ensuremath{d^n_{#1}}}


\newcommand{\eq}{equation}
\newcommand{\real}{\ensuremath{\mathbb R}}
\newcommand{\comp}{\ensuremath{\mathbb C}}
\newcommand{\rn}{\ensuremath{{\mathbb R}^n}}
\newcommand{\no}{\ensuremath{\nat_0}}
\newcommand{\ganz}{\ensuremath{\mathbb Z}}
\newcommand{\zn}{\ensuremath{{\mathbb Z}^n}}
\newcommand{\As}{\ensuremath{A^s_{p,q}}}
\newcommand{\Bs}{\ensuremath{B^s_{p,q}}}
\newcommand{\Fs}{\ensuremath{F^s_{p,q}}}
\newcommand{\nat}{\ensuremath{\mathbb N}}
\newcommand{\Om}{\ensuremath{\Omega}}
\newcommand{\di}{\ensuremath{{\mathrm d}}}
\newcommand{\Cc}{\ensuremath{\mathcal C}}
\newcommand{\vp}{\ensuremath{\varphi}}
\newcommand{\hra}{\ensuremath{\hookrightarrow}}
\newcommand{\supp}{\ensuremath{\mathrm{supp \,}}}
\newcommand{\ve}{\ensuremath{\varepsilon}}
\newcommand{\vk}{\ensuremath{\varkappa}}
\newcommand{\vr}{\ensuremath{\varrho}}
\newcommand{\pa}{\ensuremath{\partial}}
\newcommand{\id}{\ensuremath{\mathrm{id}}}
\newcommand{\pr}{\pageref}
\newcommand{\wh}{\ensuremath{\widehat}}
\newcommand{\wt}{\ensuremath{\widetilde}}
\newcommand{\os}{\ensuremath{\overset}}
\newcommand{\rank}{\ensuremath{\mathrm{rank \,}}}
\newcommand{\SRn}{\mathcal{S}(\rn)}
\newcommand{\SpRn}{\mathcal{S}'(\rn)}
\newcommand{\Dd}{\mathrm{D}}
\newcommand{\Ft}{\mathcal{F}}
\newcommand{\Fti}{\mathcal{F}^{-1}}
\newcommand{\bli}{\begin{enumerate}[{\upshape\bfseries (i)}]}
  \newcommand{\eli}{\end{enumerate}}
\newcommand{\nd}{\ensuremath {{n}}} 
\newcommand{\Ae}{\ensuremath{A^{s_1}_{p_1,q_1}}}  
\newcommand{\Az}{\ensuremath{A^{s_2}_{p_2,q_2}}}  
\newcommand{\open}[1]{\smallskip\noindent\fbox{\parbox{\textwidth}{\color{blue}\bfseries\begin{center}
      #1 \end{center}}}\\ \smallskip}
\newcommand{\red}[1]{{\color{red}#1}}
\renewcommand{\red}[1]{{#1}}
\newcommand{\blue}[1]{{\color{blue}#1}}
\renewcommand{\blue}[1]{{#1}}
\newcommand{\purple}[1]{{\color{purple}#1}}
\renewcommand{\purple}[1]{{#1}}
\newcommand{\cyan}[1]{{\color{cyan}#1}}
\newcommand{\green}[1]{{\color{green}#1}}
\newcommand{\unterbild}[1]{{\noindent\refstepcounter{figure}\upshape\bfseries
    Figure {\thefigure}{\label{#1}}}
}%
\newcommand{\ignore}[1]{}

\title{Mapping properties of Fourier transforms, revisited}
\author{Dorothee D. Haroske\footnotemark[1], Leszek Skrzypczak\footnotemark[1], and Hans Triebel}
\footnotetext[1]{The first and second author were partially supported by the German Research Foundation (DFG), Grant no. Ha 2794/8-1.}
\maketitle

\begin{abstract}
{The paper deals with continuous and compact mappings generated by the Fourier transform between distinguished Besov spaces $B^s_p(\rn)
  = B^s_{p,p} (\rn)$, $1\le p \le \infty$, and between Sobolev spaces $H^s_p (\rn)$, $1<p< \infty$. In contrast to the paper {\em H. Triebel, Mapping properties of Fourier transforms. Z. Anal. Anwend. 41 (2022), 133--152}, 
based
mainly on embeddings between related weighted spaces, we rely on wavelet expansions, duality and interpolation of corresponding
(unweighted) spaces, and (appropriately extended) Hausdorff-Young inequalities. The degree of compactness will be measured in terms of entropy numbers and approximation numbers,  now using the
symbiotic relationship to weighted spaces.}      
\end{abstract}



\ignore{

\documentclass{amse-new}

\usepackage{latexsym,mathtools}
\usepackage{color}
\usepackage{tikz-cd}
\numberwithin{equation}{section}
\usepackage{enumerate}




\begin{document}

 \PageNum{1}
 \Volume{202x}{Sep.}{x}{x}
 \OnlineTime{August 15, 202x}
 \DOI{0000000000000000}
 \EditorNote{Received x x, 202x, accepted x x, 202x}

\abovedisplayskip 6pt plus 2pt minus 2pt \belowdisplayskip 6pt
plus 2pt minus 2pt
\allowdisplaybreaks[4]

\newcommand{\SBsp}{\ensuremath{X^s_pB}}
\newcommand{\SBxp}[1]{\ensuremath{X^{#1}_pB}}
\newcommand{\SBxx}[2]{\ensuremath{X^{#1}_{#2}B}}
\newcommand{\SHsp}{\ensuremath{X^s_pH}}
\newcommand{\SHxp}[1]{\ensuremath{X^{#1}_pH}}
\newcommand{\SHxx}[2]{\ensuremath{X^{#1}_{#2}H}}
\newcommand{\TBsp}{\ensuremath{Y^s_pB}}
\newcommand{\TBxp}[1]{\ensuremath{Y^{#1}_pB}}
\newcommand{\TBxx}[2]{\ensuremath{Y^{#1}_{#2}B}}
\newcommand{\THsp}{\ensuremath{Y^s_pH}}
\newcommand{\THxp}[1]{\ensuremath{Y^{#1}_pH}}
\newcommand{\THxx}[2]{\ensuremath{Y^{#1}_{#2}H}}
\newcommand{\SAxp}[1]{\ensuremath{X^{#1}_pA}}
\newcommand{\TAxp}[1]{\ensuremath{Y^{#1}_pA}}
\newcommand{\snp}{\ensuremath{x^n_p}}
\newcommand{\tnp}{\ensuremath{y^n_p}}
\newcommand{\dnp}{\ensuremath{d^n_p}}
\newcommand{\dnx}[1]{\ensuremath{d^n_{#1}}}


\newcommand{\eq}{equation}
\newcommand{\real}{\ensuremath{\mathbb R}}
\newcommand{\comp}{\ensuremath{\mathbb C}}
\newcommand{\rn}{\ensuremath{{\mathbb R}^n}}
\newcommand{\no}{\ensuremath{\nat_0}}
\newcommand{\ganz}{\ensuremath{\mathbb Z}}
\newcommand{\zn}{\ensuremath{{\mathbb Z}^n}}
\newcommand{\As}{\ensuremath{A^s_{p,q}}}
\newcommand{\Bs}{\ensuremath{B^s_{p,q}}}
\newcommand{\Fs}{\ensuremath{F^s_{p,q}}}
\newcommand{\nat}{\ensuremath{\mathbb N}}
\newcommand{\Om}{\ensuremath{\Omega}}
\newcommand{\di}{\ensuremath{{\mathrm d}}}
\newcommand{\Cc}{\ensuremath{\mathcal C}}
\newcommand{\vp}{\ensuremath{\varphi}}
\newcommand{\hra}{\ensuremath{\hookrightarrow}}
\newcommand{\supp}{\ensuremath{\mathrm{supp \,}}}
\newcommand{\ve}{\ensuremath{\varepsilon}}
\newcommand{\vk}{\ensuremath{\varkappa}}
\newcommand{\vr}{\ensuremath{\varrho}}
\newcommand{\pa}{\ensuremath{\partial}}
\newcommand{\id}{\ensuremath{\mathrm{id}}}
\newcommand{\pr}{\pageref}
\newcommand{\wh}{\ensuremath{\widehat}}
\newcommand{\wt}{\ensuremath{\widetilde}}
\newcommand{\os}{\ensuremath{\overset}}
\newcommand{\rank}{\ensuremath{\mathrm{rank \,}}}
\newcommand{\SRn}{\mathcal{S}(\rn)}
\newcommand{\SpRn}{\mathcal{S}'(\rn)}
\newcommand{\Dd}{\mathrm{D}}
\newcommand{\Ft}{\mathcal{F}}
\newcommand{\Fti}{\mathcal{F}^{-1}}
\newcommand{\bli}{\begin{enumerate}[{\upshape\bfseries (i)}]}
  \newcommand{\eli}{\end{enumerate}}
\newcommand{\nd}{\ensuremath {{n}}} 
\newcommand{\Ae}{\ensuremath{A^{s_1}_{p_1,q_1}}}  
\newcommand{\Az}{\ensuremath{A^{s_2}_{p_2,q_2}}}  
\newcommand{\open}[1]{\smallskip\noindent\fbox{\parbox{\textwidth}{\color{blue}\bfseries\begin{center}
      #1 \end{center}}}\\ \smallskip}
\newcommand{\red}[1]{{\color{red}#1}}
\renewcommand{\red}[1]{{#1}}
\newcommand{\purple}[1]{{\color{purple}#1}}
\renewcommand{\purple}[1]{{#1}}
\newcommand{\unterbild}[1]{{\noindent\refstepcounter{figure}\upshape\bfseries
    Figure {\thefigure}{\label{#1}}}
}%
\newcommand{\ignore}[1]{}
}




\section{Introduction}   \label{S1}
Let $\Ft$,
\begin{\eq}   \label{1.1}
\big(\Ft \vp\big)(\xi) = (2 \pi)^{-n/2} \int_{\rn} e^{-i x \xi} \, \vp (x) \, \di x, \qquad \vp \in \SRn, \quad \xi \in \rn,
\end{\eq}
be the classical Fourier transform, extended  in the usual way to $\SpRn$, $n\in \nat$. The mapping properties
\begin{\eq}   \label{1.2}
\Ft\SRn = \SRn, \qquad \Ft\SpRn = \SpRn,
\end{\eq}
and
\begin{\eq}   \label{1.3}
\Ft:  L_p (\rn) \hra L_{p'} (\rn), \quad  1\le p \le 2, \quad \frac{1}{p} + \frac{1}{p'} =1, \quad \Ft L_2 (\rn) = L_2 (\rn),
\end{\eq}
are cornerstones of Fourier Analysis. We dealt in \cite{Tri22a} with further mapping properties typically of type 
\begin{\eq}   \label{1.4}
\Ft: \quad B^{s_1}_p (\rn) \hra B^{s_2}_p (\rn), \qquad 1<p<\infty,
\end{\eq}
where
\begin{\eq}   \label{1.5}
B^s_p (\rn) = B^s_{p,p} (\rn), \qquad 0<p \le \infty, \quad s \in \real,
\end{\eq}
are distinguished Besov spaces. Let, in slight modification of the notation in \cite{Tri22a},
\begin{\eq}   \label{1.7}
\SBsp (\rn) = B^{\snp + s}_p (\rn) \quad \text{and} \quad \TBsp (\rn) = B^{\tnp -s}_p (\rn),
\end{\eq}
where $n\in \nat$, $1\le p \le \infty$, and $s\in \real$, and 
\begin{\eq}  \label{1.6}
\dnp = 2n \left( \frac{1}{p} - \frac{1}{2} \right) \quad \text{and} \quad \snp = \max (0, \dnp), \quad \tnp = \min(0, \dnp),
\end{\eq}
as indicated in Figure~\ref{fig-1} below. We prove in Section~\ref{S3} that
\begin{\eq}   \label{1.8}
\Ft: \quad \SBxp{s_1} (\rn) \hra \TBxp{s_2} (\rn), \qquad 1\le p \le \infty,
\end{\eq}
is continuous if, and only if, both $s_1 \ge 0$, $s_2 \ge 0$, and that this mapping is compact if, and only if, both $s_1 >0$, 
$s_2 >0$. This modifies corresponding assertions in \cite{Tri22a}, including now limiting cases and based on new proofs. We justify
that these assertions remain valid if one replaces $B$ in \eqref{1.7} by $H$, where $H^\sigma_p (\rn)$, $1 <p<\infty$, $\sigma \in
\real$, are the (fractional) Sobolev spaces. In Section~\ref{S4} we deal with the degree of compactness of the compact mappings $\Ft$
in \eqref{1.8} in terms of entropy numbers and approximation numbers reduced to related assertions in terms of weighted spaces.

\section{Preliminaries}    \label{S2}
\subsection{Definitions}   \label{S2.1}
We use standard notation. Let $\nat$ be the collection of all natural numbers and $\no = \nat \cup \{0 \}$. Let $\rn$ be Euclidean $n$-space where $n\in \nat$. Put $\real = \real^1$, whereas $\comp$ stands for the complex plane.
Let $\SRn$ be the Schwartz space of all complex-valued rapidly decreasing infinitely differentiable functions on $\rn$, and let $\SpRn$ be the dual space of all tempered distributions on \rn.
Furthermore, $L_p (\rn)$ with $0< p \le \infty$, is the standard complex quasi-Banach space with respect to the Lebesgue measure, quasi-normed by
\begin{\eq}   \label{2.1}
\| f \, | L_p (\rn) \| = \Big( \int_{\rn} |f(x)|^p \, \di x \Big)^{1/p}
\end{\eq}
with the obvious modification if $p=\infty$.  
As usual, $\ganz$ is the collection of all integers, and $\zn$ where $n\in \nat$ denotes the
lattice of all points $m= (m_1, \ldots, m_n) \in \rn$ with $m_k \in \ganz$. Furthermore, $a_i \sim b_i$ (equivalence), $i\in I$ 
(arbitrary index set), for two sets of non-negative real numbers $\{a_i: i\in I\}$ and $\{b_i: i\in I\}$ means that there are two 
positive numbers $c_1$ and $c_2$ such that $c_1 a_i \le b_i \le c_2 a_i$ for all $i\in I$.

If $\vp \in \SRn$, then
\begin{\eq}  \label{2.2}
\wh{\vp} (\xi) = (\Ft \vp)(\xi) = (2\pi )^{-n/2} \int_{\rn} e^{-ix \xi} \vp (x) \, \di x, \qquad \xi \in  \rn,
\end{\eq}
denotes the Fourier transform of \vp. As usual, $\Fti \vp$ and $\vp^\vee$ stand for the inverse Fourier transform, given by the right-hand side of
\eqref{2.2} with $i$ in place of $-i$. Here $x \xi$ stands for the scalar product in \rn. Both $\Ft$ and $\Fti$ are extended to $\SpRn$ in the
standard way. Let $\vp_0 \in \SRn$ with
\begin{\eq}   \label{2.3}
\vp_0 (x) =1 \ \text{if $|x|\le 1$} \quad \text{and} \quad \vp_0 (x) =0 \ \text{if $|x| \ge 3/2$},
\end{\eq}
and let
\begin{\eq}   \label{2.4}
\vp_k (x) = \vp_0 (2^{-k} x) - \vp_0 (2^{-k+1} x ), \qquad x\in \rn, \quad k\in \nat.
\end{\eq}
Since
\begin{\eq}   \label{2.5}
\sum^\infty_{j=0} \vp_j (x) =1 \qquad \text{for} \quad x\in \rn,
\end{\eq}
$\vp =\{ \vp_j \}^\infty_{j=0}$ forms a dyadic resolution of unity. The entire analytic functions $(\vp_j \wh{f} )^\vee (x)$ make sense pointwise in $\rn$ for any $f\in \SpRn$. 

\begin{definition}   \label{D2.1}
Let $\vp = \{ \vp_j \}^\infty_{j=0}$ be the above dyadic resolution  of unity. Let 
$s\in \real$ and $0<p,q \le \infty$ $($with $p<\infty$ for the $F$-spaces$)$.
Then $\Bs (\rn)$ is the collection of all $f \in \SpRn$ such that
\begin{\eq}   \label{2.6}
\| f \, | \Bs (\rn) \|_{\vp} = \Big( \sum^\infty_{j=0} 2^{jsq} \big\| (\vp_j \widehat{f})^\vee \, | L_p (\rn) \big\|^q \Big)^{1/q}
\end{\eq}
is finite, whereas $\Fs (\rn)$ is the collection of all $f \in \SpRn$ such that
\begin{\eq}   \label{2.7}
\| f \, |\Fs (\rn) \|_{\vp} = \Big\| \Big( \sum^\infty_{j=0} 2^{jsq} \big| \big( \vp_j \wh{f} \big)^\vee (\cdot) \big|^q \Big)^{1/q}
\, | L_p (\rn) \Big\|
\end{\eq}
is finite $($with the usual modification if $q= \infty)$. 
\end{definition}

\begin{remark}   \label{R2.2}
These are the well-known inhomogeneous function spaces. They are independent of the above resolution of unity $\vp$ according to 
\eqref{2.3}--\eqref{2.5} (equivalent quasi-norms indicated below by $\sim$). 
This justifies the omission of the subscript $\vp$ in \eqref{2.6} and \eqref{2.7} in the sequel.
In \cite{T20} one finds (historical) references, explanations and discussions. 
We rely mainly on the special cases
\begin{align}   \label{2.8}
B^s_p (\rn) & = B^s_{p,p} (\rn) \qquad 0<p \le \infty, \quad s\in \real,\\
 \label{2.9}
\Cc^s (\rn) & = B^s_\infty (\rn), \qquad s\in \real, 
\intertext{and}
\label{2.10}
H^s_p (\rn) &= F^s_{p,2} (\rn), \qquad 1<p<\infty, \quad s\in \real.
\end{align}
Here $B^s_p (\rn)$ are special Besov spaces, including the H\"{o}lder-Zygmund spaces $\Cc^s (\rn)$, whereas $H^s_p (\rn)$ are 
(fractional) Sobolev spaces which can be equivalently normed by
\begin{\eq}   \label{2.11}
\| f \, | H^s_p (\rn) \| \sim \big\| \big( (1+|\xi|^2 )^{s/2} \wh{f} \big)^\vee \, | L_p (\rn) \big\|, \quad s\in \real.
\end{\eq}
We need later on the following refinement of \eqref{2.11}. Let $s>0$ and $1<p<\infty$. Then $H^s_p (\rn)$ is the collection of all
$f\in L_p (\rn)$ such that
\begin{\eq}   \label{2.12}
\begin{aligned}
\| f \, | H^s_p (\rn) \| &\sim \|f\,|L_p (\rn) \| + \big\| \big(|\xi|^s \wh{f} \, \big)^\vee | L_p (\rn)\| \\
&\sim \| f \, | L_p (\rn) \| + \big\| (- \Delta)^{s/2} f \, | L_p (\rn) \big\|
\end{aligned}
\end{\eq}
is finite (equivalent norms),
where the second equivalence may be considered as the fashionable re-writing of the first one. If  $s$ is not an even natural
number, then \eqref{2.12} requires some technical explanations and justifications. This may be found in \cite[Proposition 4.1, 
p.\,135]{T13}. \red{Note that for $p=2$, $B^s_2(\rn)=H^s_2(\rn) =H^s(\rn)$.}
\end{remark}

We repeat and complement what has already been said in the Introduction. Let $n\in \nat$, $1\le p \le \infty$,
\begin{\eq}  \label{2.13}
\dnp = 2n \left( \frac{1}{p} - \frac{1}{2} \right) \quad \text{and} \quad \snp = \max (0, \dnp), \quad \tnp = \min(0, \dnp),
\end{\eq}
as indicated in the figure below which is a notational modification of \cite[Figure 1]{HST22}. For convenience, we sketch the usual $(\frac1p,s)$-diagram in Figure~\ref{fig-1} where any space $B^s_p(\rn)$ is indicated by its parameters $s$ and $p$. \red{Obviously, $d^n_2 = x^n_2 = y^n_2 = 0$.}\\

\begin{definition}    \label{D2.3}
  Let $n\in \nat$ and $s\in \real$.
  \bli
  \item
Let $1\le p \le \infty$. Then
\begin{\eq}   \label{2.14}
\SBsp (\rn) = B^{\snp +s}_p (\rn) \quad \text{and} \quad \TBsp (\rn) = B^{\tnp -s}_p (\rn).
\end{\eq}
\item Let $1<p< \infty$. Then
\begin{\eq}   \label{2.15}
\SHsp (\rn) = H^{\snp +s}_p (\rn) \quad \text{and} \quad \THsp (\rn) = H^{\tnp -s}_p (\rn).
\end{\eq}
\eli
\end{definition}

\begin{remark}   \label{R2.4}
This modifies and extends \cite[Definition 2.3]{Tri22a}. As there it comes out that $|\dnp|$ is a natural gap if one asks for 
continuous mappings
\begin{\eq}   \label{2.16}
\Ft: \quad \SAxp{s_1} (\rn) \hra \TAxp{s_2} (\rn), \qquad A  \in \{B, H \},
\end{\eq}
for fixed $p$ between the {source spaces} $\SAxp{s_1} (\rn)$ and the {target spaces} $\TAxp{s_2} (\rn)$, what may also explain
the notation $\snp$ and $\tnp$ in \eqref{2.13}. \red{In case of $p=2$, then
  \begin{equation}\label{2.16a}
    \SBxx{s}{2}(\rn) = B^s_2(\rn)=H^s_2(\rn)=\SHxx{s}{2}(\rn), \quad \TBxx{s}{2}(\rn)=B^{-s}_2(\rn)=H^{-s}_2(\rn) = \THxx{s}{2}(\rn).
  \end{equation}}%
In the present paper  we deal mainly with unweighted spaces, both in
the formulation of assertions and in the related proofs at least as far as continuous mappings in Section~\ref{S3} are concerned.
But it is quite clear now that a comprehensive theory should also cover 
corresponding weighted spaces with so-called admissible weights $w_\alpha (x) = (1 + |x|^2 )^{\alpha/2}$, $x\in \rn$, $\alpha \in 
\real$. We used already weighted spaces in \cite{Tri22a} in the proofs (but not in the formulation of results). In \cite{Tri22b} we 
indicated how assertions can be extended from unweighted spaces to their weighted counterparts. We refer the reader in this context
also to \cite{HST22} where we dealt with nuclear mappings  generated by the Fourier transform. In the present paper we measure the
degree of compactness  in terms of entropy numbers and approximation numbers. Then we rely on already available  assertions for
embeddings between related weighted spaces. The outcome may justify the new notation in \eqref{2.14} and \eqref{2.15}. 
\end{remark}

\noindent\begin{minipage}{\textwidth}
  ~\hfill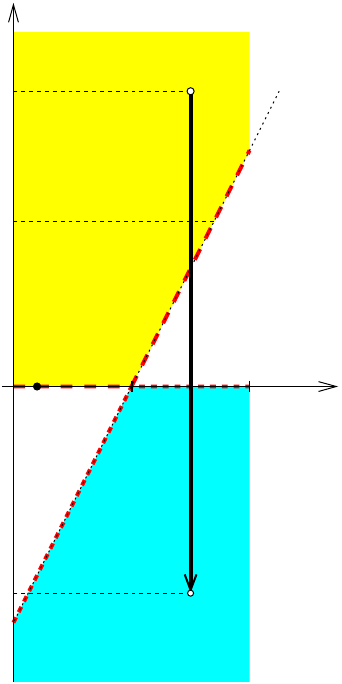\hfill~\\
~\hspace*{\fill}\unterbild{fig-1}
\end{minipage}
\smallskip~

\subsection{Wavelet characterizations}    \label{S2.2}
We assume that the reader is familiar with the basic assertions for the spaces $\As (\rn)$, $A \in \{B,F \}$, according to Definition
\ref{D2.1}, including wavelet characterizations \purple{using Daubechies-type wavelets, cf. \cite{Dau}}. It is sufficient for what follows to concentrate on the $B$-spaces. We
follow  otherwise
\cite[Section 1.2.1, pp.\,7--10]{T20}. There one finds explanations, discussions and, in particular, \purple{further} references. This will not
be repeated here. As usual, $C^{u} (\real)$ with $u\in
\nat$ collects all bounded complex-valued continuous functions on $\real$ having continuous bounded derivatives up to order $u$ inclusively. Let
\begin{\eq}   \label{2.17}
\psi_F \in C^{u} (\real), \qquad \psi_M \in C^{u} (\real), \qquad u \in \nat,
\end{\eq}
be {\em real} compactly supported Daubechies wavelets with
\begin{\eq}   \label{2.18}
\int_{\real} \psi_M (x) \, x^v \, \di x =0 \qquad \text{for all $v\in \no$ with $v<u$.}
\end{\eq}
One extends these wavelets from $\real$ to $\rn$ by the usual tensor product procedure. Let $n\in \nat$ and let
\begin{\eq}   \label{2.19}
G = (G_1, \ldots, G_n) \in G^0 = \{F,M \}^n
\end{\eq}
which means that $G_r$ is either $F$ or $M$. Furthermore, let
\begin{\eq}   \label{2.20}
G= (G_1, \ldots, G_n) \in G^* = G^j = \{F, M \}^{n*}, \qquad j \in \nat,
\end{\eq}
which means that $G_r$ is either $F$ or $M$, where $*$ indicates that at least one of the components of $G$ must be an $M$. Hence $G^0$ has $2^n$ elements, whereas $G^j$ with $j\in \nat$ and $G^*$ have $2^n -1$ elements. Let
\begin{\eq}   \label{2.21}
\psi^j_{G,m} (x) = \prod^n_{l=1} \psi_{G_l} \big(2^j x_l -m_l \big), \qquad G\in G^j, \quad m \in \zn, \quad x\in \rn,
\end{\eq}
where (now) $j \in \no$. We always assume that $\psi_F$ and $\psi_M$ in \eqref{2.17} have $L_2$ norm 1. Then 
\begin{\eq}   \label{2.22}
 \big\{ 2^{jn/2} \psi^j_{G,m}: \ j \in \no, \ G\in G^j, \ m \in \zn \big\}
\end{\eq}
is an orthonormal basis in $L_2 (\rn)$ (for any $u\in \nat$) and 
\begin{\eq}   \label{2.23}
f = \sum^\infty_{j=0} \sum_{G \in G^j} \sum_{m \in \zn} \lambda^{j,G}_m \, \psi^j_{G,m}
\end{\eq}
with
\begin{\eq}   \label{22.24}
\lambda^{j,G}_m = \lambda^{j,G}_m (f) = 2^{jn} \int_{\rn} f(x) \, \psi^j_{G,m} (x) \, \di x = 2^{jn} \big(f, \psi^j_{G,m} \big)
\end{\eq}
is the corresponding expansion. For our purpose it will be sufficient to describe how wavelet expansions of $f\in \Bs (\rn)$ with
$0<p,q \le \infty$ and $s\in \real$ look like. Let
\begin{\eq}   \label{2.25}
\lambda = \big\{ \lambda^{j,G}_m \in \comp: \ j \in \no, \ G \in G^j, \ m \in \zn \big\}
\end{\eq}
and
\begin{\eq}   \label{2.26}
b^s_{p,q} (\rn) = \big\{ \lambda: \ \| \lambda \, | b^s_{p,q} (\rn) \| < \infty \big\}
\end{\eq}
with
\begin{\eq}   \label{2.27}
\| \lambda \, | b^s_{p,q} (\rn) \| = \Big( \sum^\infty_{j=0} 2^{j(s- \frac{n}{p})q} \sum_{G \in G^j} \Big( \sum_{m \in \zn} |\lambda^{j,G}_m|^p \Big)^{q/p} \Big)^{1/q}
\end{\eq}
(usual modification if $\max(p,q) = \infty$). Let
\begin{\eq}   \label{2.28}
\sigma^n_p = n \Big( \max \big( \frac{1}{p}, 1 \big) - 1 \Big), \qquad 0<p \le \infty.
\end{\eq}

\begin{proposition}   \label{P2.5}
Let $0<p \le \infty$, $0<q \le \infty$, $s\in \real$ and
\begin{\eq}   \label{2.29}
u > \max (s, \sigma^n_p -s).
\end{\eq}
Let $f \in \SpRn$. Then $f \in \Bs (\rn)$ if, and only if, it can be represented as
\begin{\eq}   \label{2.30}
f= \sum_{\substack{j\in \no, G\in G^j, \\ m\in \zn}}
 \lambda^{j,G}_m \, \psi^j_{G,m}, \qquad \lambda \in b^s_{p,q} (\rn),
\end{\eq}
the unconditional convergence being in $\SpRn$. The representation \eqref{2.30} is unique,
\begin{\eq}  \label{2.31}
\lambda^{j,G}_m = \lambda^{j,G}_m (f) =  2^{jn} \big( f, \psi^j_{G,m} \big)
\end{\eq}
and
\begin{\eq}   \label{2.32}
I: \quad f \mapsto \big\{ \lambda^{j,G}_m (f) \big\}
\end{\eq}
is an isomorphic mapping of $\Bs (\rn)$ onto $b^s_{p,q} (\rn)$.
\end{proposition}

\begin{remark}   \label{R2.6}
This coincides with the relevant part of \cite[Proposition 1.11, pp.\,9--10]{T20}. There one finds related references. 
\end{remark}

\begin{remark}   \label{R2.7}
The representation \eqref{2.30}, \eqref{2.31} can be extended to $f\in \SpRn$. This can be justified as follows. Let $w_\alpha (x) =
(1 + |x|^2)^{\alpha/2}$, $x\in \rn$, $\alpha \in \real$, be a so-called admissible weight. Let $s\in \real$ and $0< p,q \le \infty$.
Then the weighted Besov space $\Bs (\rn, w_\alpha)$ collects all $f\in \SpRn$ such that
\begin{\eq}   \label{2.33}
\| f \, |\Bs (\rn, w_\alpha) \| = \| w_\alpha f \, | \Bs (\rn) \|
\end{\eq}
is finite. For fixed $p$ and $q$ one has
\begin{\eq}   \label{2.34}
\SpRn = \bigcup_{\alpha,s \in \real} \Bs (\rn, w_\alpha)
\end{\eq}
according to \cite[Remark 2.91, p.\,74]{T20} with a reference to \cite{Kab08}. On the other hand one can extend the above proposition
to these weighted Besov spaces $\Bs (\rn, w_\alpha)$. We refer the reader to \cite[Theorem~6.15, pp.\,270--271]{T06}. This shows that
expansions of type \eqref{2.30}, \eqref{2.31} make sense for any $f\in \SpRn$.
\end{remark}

\section{Mapping properties}   \label{S3}
\subsection{Besov spaces}    \label{S3.1}
Let $A_1 (\rn)$ and $A_2 (\rn)$ be two function spaces as introduced in Definition~\ref{D2.1} and let $\Ft$ be again the above Fourier
transform. Then $\Ft: A_1 (\rn) \hra A_2 (\rn)$ means that $\Ft$ (better its restriction from $\SpRn$ to $A_1 (\rn))$ is a linear and
continuous mapping from $A_1 (\rn)$ into $A_2 (\rn)$. Otherwise we use the above notation, including the modifications in Definition 
\ref{D2.3} adapted to mapping properties of the Fourier transform.

\begin{theorem}   \label{T3.1}
Let $1 \le p \le \infty$ and $s_1 \in \real$, $s_2 \in \real$. Then there is a continuous mapping 
\begin{\eq}   \label{3.1}
\Ft: \quad \SBxp{s_1} (\rn) \hra \TBxp{s_2} (\rn)
\end{\eq}
if, and only if, both $s_1 \ge 0$ and $s_2 \ge 0$. This mapping is compact if, and only if, both $s_1 >0$ and $s_2 >0$.
\end{theorem}

\begin{proof}
{\em Step 1.} First we deal with $p=\infty$ and $s_1 = s_2 =0$. Then \eqref{3.1} reduces to
\begin{\eq}   \label{3.2}
\Ft: \quad \Cc^0 (\rn) \hra \Cc^{-n} (\rn),
\end{\eq}
where we used \eqref{2.9}. Let $f \in \Cc^0 (\rn)$. Then it follows from Proposition~\ref{P2.5} and Remark~\ref{R2.7} that $\wh{f} \in
\SpRn$ admits the wavelet expansion
\begin{\eq}   \label{3.3}
\wh{f}= \sum_{\substack{j\in \no, G\in G^j, \\ m\in \zn}}
 \lambda^{j,G}_m  (\wh{f}) \, \psi^j_{G,m} = \sum_{\substack{j\in \no, G\in G^j, \\ m\in \zn}} 2^{jn} (\wh{f}, \psi^j_{G,m}) \,
\psi^j_{G,m},
\end{\eq}
unconditional convergence being in $\SpRn$. We rely on the corresponding wavelet expansion for $g\in \Cc^s (\rn)$, $s\in \real$, with
\begin{\eq}  \label{3.4}
\| g \, |\Cc^s (\rn) \| \sim \sup_{\substack{j\in \no, G\in G^j, \\ m\in \zn}} 2^{j(s+n)} | (g, \psi^j_{G,m} )|.
\end{\eq}
Let $f\in \Cc^0 (\rn)$. Then it follows from the duality $B^0_1 (\rn)' = \Cc^0 (\rn)$ according to \cite[Theorem~2.11.2, p.\,178]{T83}
and $(\wh{f}, \psi^j_{G,m} ) = (f, \wh{\psi^j_{G,m}} )$ (what is, almost, the definition of $\Ft$ extended from $\SRn$ to $\SpRn$)
that
\begin{\eq}   \label{3.5}
\big| (\wh{f}, \psi^j_{G,m} ) \big| \le c \, \|  f\, |\Cc^0 (\rn)\| \cdot \| \wh{\psi^j_{G,m}} \, | B^0_1 (\rn) \|.
\end{\eq}
Inserted in \eqref{2.6} one has by the support properties of the functions involved,
\begin{\eq}   \label{3.6}
\| \wh{\psi^j_{G,m}} \, | B^0_1 (\rn) \| \sim \| \wh{\psi^j_{G,m}} \, |L_1 (\rn) \|.
\end{\eq}
Using
\begin{\eq}   \label{3.7}
\wh{\psi^j_{G,m}} (\xi) = 2^{-jn} e^{-i 2^{-j} m \xi} \,\wh{\psi_G} (2^{-j} \xi), \qquad \xi \in \rn,
\end{\eq}
(in obvious notation) one obtains that for  some $c>0$,
\begin{\eq}   \label{3.8}
\| \wh{\psi^j_{G,m}} \, | B^0_1 (\rn) \| \le c, \qquad j\in \no, \quad G\in G^j, \quad m\in \zn.
\end{\eq}
Now \eqref{3.2} follows from \eqref{3.4}, \eqref{3.5}.
\\

{\em Step 2.} We prove
\begin{\eq}   \label{3.9}
\Ft: \quad \SBxp{0} (\rn) = B^0_p (\rn) \hra B^{\dnp}_p (\rn) = \TBxp{0} (\rn), \qquad 2 \le p \le \infty,
\end{\eq}
by complex interpolation of \eqref{3.2},
\begin{\eq}   \label{3.10}
\Ft: \quad \SBxx{0}{\infty} (\rn) = \Cc^0 (\rn) \hra \Cc^{-n}(\rn) = \TBxx{0}{\infty} (\rn)
\end{\eq}
and
\begin{\eq}   \label{3.11}
\Ft: \quad \SBxx{0}{2} (\rn) = L_2 (\rn) \hra L_2 (\rn) = \TBxx{0}{2} (\rn),
\end{\eq}
where we used the notation as introduced in \eqref{2.13}, \eqref{2.14}, in particular $\dnp = 2n \big( \frac{1}{p} - \frac{1}{2} 
\big)$. The wavelet isomorphism $I$ in \eqref{2.32} shows that it is sufficient to interpolate the corresponding sequence spaces,
\begin{\eq}   \label{3.12}
\big[ b^0_{2,2} (\rn), b^0_{\infty, \infty} (\rn) \big]_\theta = 
b^0_{p,p} (\rn), \qquad \frac{1}{p} = \frac{1-\theta}{2}, \quad 2<p<\infty,
\end{\eq}
and
\begin{\eq}    \label{3.13}
\big[b^0_{2,2} (\rn), b^{-n}_{\infty, \infty} (\rn) \big]_\theta =
b^{\dnp}_{p,p} (\rn), \qquad \frac{1}{p} = \frac{1-\theta}{2}, \quad 2<p<\infty,
\end{\eq}
where $0<\theta <1$. One has in obvious notation
\begin{\eq}   \label{3.14}
b^0_{2,2} (\rn) = \ell_2 \big( 2^{- \frac{jn}{2}} \ell_2 \big),
\end{\eq}
\begin{\eq}   \label{3.15}
b^0_{\infty, \infty} (\rn) = {\ell_{\infty} \big(
\ell_\infty \big) } \quad \text{and} \quad b^{-n}_{\infty, \infty} (\rn) = \ell_\infty \big( 2^{-jn}
\ell_\infty \big).
\end{\eq}
We use the complex interpolation
\begin{\eq}   \label{3.16}
\big[ \ell_2 (A_j), \ell_\infty (B_j) \big]_\theta = \ell_p \big([A_j, B_j]_\theta \big), \quad \frac{1}{p} = \frac{1-\theta}{2},
\end{\eq}
according to \cite[Remark 1.18.1/2, (12), p.\,122]{T78}, where $\{A_j, B_j \}^\infty_{j=1}$ are interpolation couples, specified to
the above weighted sequence spaces. Then  one obtains \eqref{3.12}, \eqref{3.13} from \eqref{3.14}--\eqref{3.16} and
\begin{\eq}   \label{3.17}
\big[ 2^{-\frac{jn}{2}} \ell_2, \ell_\infty \big]_\theta = 2^{-\frac{jn}{p}} \ell_p,
\end{\eq}
\begin{\eq}   \label{3.18}
\big[ 2^{-\frac{jn}{2}} \ell_2, 2^{-jn} \ell_\infty \big]_\theta =  2^{\frac{jn}{2} (1-\theta) - jn} \ell_p =
 2^{j(\dnp - \frac{n}{p})} \ell_p,
\end{\eq}
where again $0<\theta <1$ and $\frac{1}{p} = \frac{1-\theta}{2}$. The interpolation property shows that \eqref{3.9} follows from
\eqref{3.10} and \eqref{3.11}.
\\

{\em Step 3.} We prove
\begin{\eq}   \label{3.19}
\Ft: \quad \SBxp{0} (\rn) = B^{\dnp}_p (\rn) \hra B^0_p (\rn) = \TBxp{0} (\rn), \qquad 1 \le p \le 2,
\end{\eq}
with $\dnp = 2n \big( \frac{1}{p} - \frac{1}{2} \big) \ge 0$ by duality. The Fourier transform $\Ft$ is self-dual, $\Ft' = \Ft$, in the dual pairing
$ \big( \SRn, \SpRn \big)$ and its restriction to related function spaces. One has according to \cite[Theorem~2.11.2, Remark 2.11.2/2, 
pp.\,178, 180]{T83} that
\begin{\eq}   \label{3.20}
B^s_p (\rn)' = B^{-s}_{p'} (\rn), \qquad 1 \le p <\infty, \quad \frac{1}{p} + \frac{1}{p'} =1, \quad s \in \real,
\end{\eq}
and
\begin{\eq}   \label{3.21}
\os{\circ}{\Cc}{}^s (\rn)' = B^{-s}_1 (\rn), \qquad s\in \real,
\end{\eq}
{where $\os{\circ}{\Cc}{}^s (\rn)$  denotes the completion of  $\SRn$ in ${\Cc}{}^s (\rn)$}.   
We rely on \eqref{3.9} and the modification 
\begin{\eq}   \label{3.22}
\Ft: \quad \os{\circ}{\Cc}{}^0 (\rn) \hra \os{\circ}{\Cc}{}^{-n} (\rn)
\end{\eq}
of \eqref{3.10}. With $\dnx{p'} = - \dnp$ one has
\begin{\eq}   \label{3.23}
\TBsp (\rn)'= B^{\dnp -s}_p (\rn)' = B^{\dnx{p'} +s}_{p'} (\rn) = \SBxx{s}{p'} (\rn),
\end{\eq}
$s\in \real$, $2 \le p <\infty$, complemented by a corresponding assertion for $p =\infty$ based on \eqref{3.22}. Then \eqref{3.19} follows from
$\Ft' = \Ft$, \eqref{3.9} and \eqref{3.20}--\eqref{3.23}.
\\

{\em Step 4.} Furthermore \eqref{3.9}, \eqref{3.19} and elementary embeddings show that $\Ft$ in \eqref{3.1} is continuous if both $s_1 \ge 0$ and 
$s_2 \ge 0$. We wish to show that this mapping is not compact if either $s_1 =0$ or $s_2 =0$ and deal first with
\begin{\eq}   \label{3.24}
\Ft: \quad \SBxp{s_1} (\rn) = B^{s_1}_p (\rn) \hra B^{\dnp}_p (\rn) = \TBxp{0} (\rn), \qquad 2 \le p \le \infty,
\end{\eq}
$s_1 \ge 0$, where again $\dnp = 2n \big(\frac{1}{p} - \frac{1}{2} \big)$. Since $\dnp - \frac{n}{p} = - \frac{n}{p'}$, $\frac{1}{p} + 
\frac{1}{p'} =1$, it follows from $B^{\dnp}_p (\rn) \hra \Cc^{- \frac{n}{p'}}(\rn)$ that it is sufficient to show that
\begin{\eq}   \label{3.25}
\Ft: \quad B^{s_1}_p (\rn) \hra \Cc^{- \frac{n}{p'}} (\rn), \qquad s_1 \ge 0, \quad 2\le p \le \infty,
\end{\eq}
is not compact. We rely on the same type of arguments as in \cite{Tri22a}. Let $\psi$ be a non-trivial  compactly supported $C^\infty$ function
in $\rn$ and $\psi_j (x) = \psi(2^{-j} x )$, $x\in \rn$, $j\in \nat$, such that 
\begin{\eq}   \label{3.26}
\psi_j(x) \cdot \vp_j (x) = \psi_j (x) \quad \text{and} \quad \supp \psi_j \cap \supp \vp_k = \emptyset \quad \text{if} \quad k\in \nat, \quad j \not= k,
\end{\eq}
where $\{\vp_j \}$ is a suitably chosen resolution of unity according to \eqref{2.3}--\eqref{2.5}. Let $f_j (x) = 2^{-\frac{jn}{p}} \psi_j (x)$.
Then
\begin{\eq}   \label{3.27}
\| f_j \,| W^m_p (\rn) \| = \sum_{|\alpha| \le m} \| \Dd^\alpha f_j \, | L_p (\rn) \| \sim 1, \qquad j\in \nat,
\end{\eq}
where $W^m_p (\rn)$, $m\in \nat$, $2\le p <\infty$, are the usual Sobolev spaces, complemented by $W^m_\infty (\rn) = C^m (\rn)$, the collection of
all bounded continuous  functions in $\rn$ having bounded classical derivatives  up to order $m$ inclusively. It follows from $W^m_p (\rn) \hra
B^{s_1}_p (\rn) \hra L_p (\rn)$, $0<s_1 <m$, $2 \le p \le \infty$, that
\begin{\eq}   \label{3.28}
\| f_j \, |B^{s_1}_p (\rn) \| \sim 1, \qquad j\in \nat, \quad s_1 >0, \quad 2 \le p \le \infty.
\end{\eq}
One has by \eqref{2.6},
\begin{\eq}   \label{3.29}
\begin{aligned}
\| \wh{f_j} \, | \Cc^{-\frac{n}{p'}} (\rn)\| &\sim \sup_{k\in \no, x\in \rn} 2^{-\frac{kn}{p'}} \big| (\vp_k f_j )^\vee (x) \big| \\
&\sim \sup_{x \in \rn} 2^{- \frac{jn}{p'} - \frac{jn}{p}} \big| \psi^\vee_j (x) \big| \sim 1,
\end{aligned}
\end{\eq}
and
\begin{\eq}   \label{3.30}
\| \wh{f_j} - \wh{f_l} \, | \Cc^{-\frac{n}{p'}} (\rn) \| \sim 1 \qquad \text{if} \quad j \not= l.
\end{\eq}
This shows that $\Ft$ in \eqref{3.25} is not compact. Next we justify that also
\begin{\eq}  \label{3.31}
\Ft: \quad \SBxp{0} (\rn) = B^0_p (\rn) \hra B^{\dnp - s_2}_p (\rn) = \TBxp{s_2} (\rn), \qquad 2 \le p \le \infty,
\end{\eq}
$s_2 \ge 0$, is not compact. Again we modify corresponding arguments in \cite{Tri22a}. It follows by embedding that it is sufficient 
to show that
\begin{\eq}   \label{3.32}
\Ft: \quad B^0_p (\rn) \hra \Cc^\sigma (\rn), \qquad \sigma \le - n,
\end{\eq}
is not compact. Let $\psi = \psi_F$ be a smooth compactly supported father wavelet according to \eqref{2.17} and let $\psi_m (x) = \psi
(x-m)$, $x\in \rn$, $m\in \zn$. Let
\begin{\eq}   \label{3.33}
f_m (x) = \big( \Fti \psi_m \big) (x) = e^{imx} \Fti \psi(x), \qquad x\in \rn, \quad m\in \zn.
\end{\eq}
Then one has  by Proposition~\ref{P2.5} that $\{\Ft f_m = \psi_m: \, m\in \zn \}$ is not compact in any space $\Cc^\sigma (\rn)$. We
insert $f_m$ in \eqref{2.6} with $\Bs (\rn) = B^0_p (\rn)$. Then
\begin{\eq}   \label{3.34}
\big( \vp_j \wh{f_m} \big)^\vee (x) = (\vp_j \psi_m )^\vee (x) =0
\end{\eq}
with exception of the related terms with $|m| \sim 2^j$. We may assume $\vp_j \psi_m = \psi_m$. Then one has
\begin{\eq}   \label{3.35}
\| (\vp_j \psi_m)^\vee | L_p (\rn) \| = \| e^{imx} \psi^\vee | L_p (\rn) \| = \| \psi^\vee |L_p (\rn) \| 
\end{\eq}
and
\begin{\eq}   \label{3.36}
\| f_m \, | B^0_p (\rn) \| \sim 1, \qquad m\in \zn, \quad 2 \le p \le \infty.
\end{\eq}
This shows that $\Ft$ in \eqref{3.31} is not compact. By \eqref{3.22} it follows that also
\begin{\eq}  \label{3.37}
\Ft: \quad \os{\circ}{\Cc}{}^0 (\rn) \hra \os{\circ}{\Cc}{}^\sigma (\rn), \qquad \sigma \le -n,
\end{\eq}
is not compact.
\\

{\em Step 5.} As far as the non-compactness of the counterparts both of \eqref{3.24} and \eqref{3.31} for $1\le p \le 2$ are concerned
we rely on the well-known assertion that the linear and continuous mapping $T: A \hra B$ between the Banach spaces $A$ and $B$ is
compact if, and only if, the dual operator $T': B' \hra A'$ is compact, \cite[Theorem~4.19, p.\,105]{Rud91}. Then it follows from $\Ft'
= \Ft$, \eqref{3.20}--\eqref{3.23} and \eqref{3.37} that the above assertion about non-compactness can be transferred form $2\le p \le
\infty$ to $1\le p \le 2$. 
\\

{\em Step 6.} As far as the remaining assertions of the above theorem are concerned we can rely on what has already been proved in
\cite{Tri22a}. One has by \cite[Theorem~4.8, Remark 4.9, pp.\,146, 148]{Tri22a} that
\begin{\eq}   \label{3.38}
\Ft: \quad \SBxp{s_1} (\rn) \hra \TBxp{s_2} (\rn), \qquad 1 \le p \le 2,
\end{\eq}
is compact if both $s_1 >0$ and $s_2 >0$. The above duality argument based on \eqref{3.20} extends this assertion to $2\le p \le
\infty$ (including now $p= \infty$). Finally, if either $s_1 <0$ or $s_2 <0$, then there is no continuous mapping \eqref{3.1}. This can
be justified in the same way as in \cite[Corollary 3.3,  p.\,142]{Tri22a}, now including $p=1$ and $p= \infty$. 
\end{proof}

\begin{remark}    \label{R3.l2}
In Step~6 we relied on compactness assertions proved in \cite{Tri22a} based there on corresponding properties for compact embeddings
between weighted function spaces. One may ask for new proofs in the framework of the above tools, wavelet expansions, interpolation,
and duality. But this is not so clear. We comment  on this problem at the beginning  of Section~\ref{S4.2} where we deal with
approximation numbers.
\end{remark}

\subsection{Sobolev spaces}    \label{S3.2}
The only paper known to us which fits in the above scheme is \cite{Ryd20}. It is one of the main aims of this paper to justify the
continuous mapping
\begin{\eq}   \label{3.39}
\Ft: \quad L_p (\rn) \hra H^{\dnp}_p (\rn), \qquad 2<p<\infty,
\end{\eq}
where $H^s_p (\rn)$ are the Sobolev spaces according to \eqref{2.10}, \eqref{2.11} and again $\dnp= 2n \big( \frac{1}{p} - 
\frac{1}{2} \big)$. We give a short new proof as a first step of the counterpart of Theorem~\ref{T3.1} with $H$ in place of $B$. We 
use the notation as introduced in Definition~\ref{D2.3}. Then \eqref{3.39} can be rewritten as
\begin{\eq}   \label{3.40}
\Ft: \quad \SHxp{0} (\rn) = L_p (\rn) \hra H^{\dnp}_p (\rn) = \THxp{0} (\rn), \qquad 2<p<\infty.
\end{\eq}

\begin{theorem}   \label{T3.3}
Let $1<p<\infty$ and $s_1 \in \real$, $s_2 \in \real$. Then there is a continuous mapping
\begin{\eq}   \label{3.41}
\Ft: \quad \SHxp{s_1} (\rn) \hra \THxp{s_2}(\rn)
\end{\eq}
if, and only if, both $s_1 \ge 0$ and $s_2 \ge 0$. This mapping is compact if, and only if,  both $s_1 >0$ and $s_2 >0$.
\end{theorem}

\begin{proof}
{\em Step 1.} First we prove
\begin{\eq}   \label{3.42}
\Ft: \quad \SHxp{0} (\rn) = H^{\dnp}_p (\rn) \hra L_p (\rn) = \THxp{0}(\rn), \qquad 1<p \le 2,
\end{\eq}
where again $\dnp = 2n \big( \frac{1}{p} - \frac{1}{2} \big)$. One has according  to \cite[Theorem~1.4.1, p.\,11]{BeL76}
\begin{\eq}    \label{3.43}
\Big( \int_{\rn} | \wh{f} (x)|^p |x|^{n(p-2)} \, \di x \Big)^{1/p} \le c \, \|f \, | L_p (\rn) \|, \qquad 1<p \le 2,
\end{\eq}
(correcting $1\le p \le 2$ there into $1<p \le 2$). Inserting $\big( |\xi|^{\dnp} \wh{f} \big)^\vee$ in the understanding as already
discussed in connection with \eqref{2.12} with a reference to \cite[Proposition 4.1, p.\,135]{T13} one obtains
\begin{\eq}   \label{3.44}
\Big( \int_{\rn} |\wh{f} (x)|^p \, \di x \Big)^{1/p} \le c \, \big\| \big( |\xi|^{\dnp} \wh{f}\,\big)^\vee | L_p (\rn) \big\|.
\end{\eq}
Then \eqref{3.42} follows from \eqref{2.12}. Afterwards \eqref{3.40} is a matter of duality, $\Ft' =\Ft$,
\begin{\eq}   \label{3.45}
H^{\dnp}_p (\rn)' = H^{\dnx{p'}}_{p'} (\rn), \quad L_p (\rn)' = L_{p'} (\rn), \quad \dnx{p'} = - \dnp,
\end{\eq}
where $1<p<2$ and $\frac{1}{p} + \frac{1}{p'} =1$. Here the well-known duality $H^\sigma_p (\rn)' = H^{-\sigma}_{p'} (\rn)$, $1<p<
\infty$, $\sigma \in \real$, is a special case of \cite[Theorem~2.11.2, p.\,178]{T83} with $H^\sigma_p (\rn) = F^\sigma_{p,2} (\rn)$,
$1<p<\infty$. It is also covered by \cite[Theorem~2.6.1, pp.\,198--199]{T78}. This proves the continuous mapping
\begin{\eq}   \label{3.46}
\Ft: \quad \SHxp{0} (\rn) \hra \THxp{0}(\rn), \qquad 1<p<\infty.
\end{\eq}

{\em Step 2.} The remaining assertions of the above theorem can now be proved in the same way as in the Steps~4-6 of the proof of
Theorem~\ref{T3.1}.
\end{proof}

\begin{remark}   \label{R3.4}
Let $1<p<2$, $\frac{1}{p} + \frac{1}{p'} =1$, and again $\dnp = \frac{2n}{p} -n >0$. Then \eqref{3.44} combined with the 
Hausdorff-Young inequality \eqref{1.3} shows that
\begin{\eq}   \label{3.47}
\begin{aligned}
\|f \, | L_{p'} (\rn) \| & \le c \, \| \wh{f} \, | L_p (\rn) \| \\
&\le c' \, \big\| (|\xi|^{\dnp} \wh{f} )^\vee | L_p (\rn) \big\|  \\
&\le c'' \| f \, | H^{\dnp}_p (\rn) \|, \qquad f \in H^{\dnp}_p (\rn).
\end{aligned}
\end{\eq}
The spaces $L_{p'} (\rn)$ and $H^{\dnp}_p (\rn)$ have the same so-called differential dimension $-\frac{n}{p'} = \dnp - \frac{n}{p}
$: If one replaces $f(\cdot)$ by $f(\lambda \cdot)$, $\lambda >0$, then it comes out that the first three terms in \eqref{3.47} have
the same homogeneity $\lambda^{- \frac{n}{p'}}$. This shows that \eqref{3.47} is a refinement of the related limiting embedding
\begin{\eq}   \label{3.48}
\id: \quad H^{\dnp}_p (\rn) \hra L_{p'} (\rn), \qquad 1<p<2.
\end{\eq}
\end{remark}

\subsection{Weighted spaces and weighted  inequalities}
The elementary properties of the Fourier transform  show that it converts smoothness into weights and vice versa. This suggests  to extend some assertions about mappings between the unweighted spaces  to their weighted counterparts. Here we recall inequalities of Hausdorff-Young type  for the so--called admissible weights
\begin{\eq}   \label{w1}
	w_\alpha (x) = (1 + |x|^2)^{\alpha/2}, \qquad \alpha \in \real, \quad x \in \rn,
\end{\eq}
that will be needed later on. The inequalities were  proved in \cite{Tri23}. We recall their proof for reader's convenience. 

Let $L_p (\rn, w_\alpha)$ with $0<p \le \infty$ and $\alpha \in \real$ be the complex quasi--Banach space quasi--normed by
\begin{\eq}   \label{w2}
	\| f \, |L_p (\rn, w_\alpha) \| = \| w_\alpha f \, | L_p (\rn) \| . 
\end{\eq}
 Then $F^s_{p,q} (\rn, w_\alpha)$  (resp. $B^s_{p,q} (\rn, w_\alpha)$) is the collection of all $f\in \SpRn$ such that \red{\eqref{2.7}}, (resp. \red{\eqref{2.6}}) with $L_p (\rn, w_\alpha)$ in place of $L_p (\rn)$ is finite. They are again quasi--Banach spaces which are independent of the chosen resolution of unity according to 
\eqref{2.3}--\eqref{2.5}. The theory of these peculiar weighted spaces
goes back to \cite[Section 4.2]{ET96} based, in turn, on \cite{HaT94}. These spaces have some remarkable properties. In particular
$f \mapsto w_\alpha f$ is an isomorphic mapping for all spaces $\As (\rn, w_\alpha)$ onto $\As (\rn)$,
\begin{\eq}   \label{w3}
	\|w_\alpha f \, | \As (\rn) \| \sim \|f \, | \As (\rn, w_\alpha) \|,
\end{\eq}
where $\As$ stands for $\Bs$ or $\Fs$. 
This shows that many properties for the unweighted spaces can be transferred to their weighted counterparts. 

\begin{theorem}   \label{T3.5new}
  \bli
  \item Let $2\le p \le \infty$, $\frac{1}{p} + \frac{1}{p'} =1$, $s\in \real$ and $\sigma \in \real$. Then $\Ft$,
	\begin{\eq}   \label{T3.5_1}
		\Ft: \quad B^s_{p'} (\rn, w_\sigma) \hra B^\sigma_p (\rn, w_s),
	\end{\eq}
	is a continuous, but not compact mapping.
\item Let $2\le p <\infty$, $\frac{1}{p} + \frac{1}{p'} =1$, $s\in \real$ and $\sigma \in \real$. Then $\Ft$,
	\begin{\eq}   \label{T3.5_2}
		\Ft: \quad H^s_{p'} (\rn, w_\sigma) \hra H^\sigma_p (\rn, w_s),
	\end{\eq}
	is a continuous, but not compact mapping.
\eli
      \end{theorem}

To prove the above theorem   we need two operators acting between (un-) weighted Sobolev and Besov spaces: the multiplication operator $W_\beta$ and the lift $I_\gamma$. The multiplication $W_\beta$ is an operator defined in the following way  
	\begin{\eq}    \label{w_mult}
		W_\beta: \quad f \mapsto w_\beta f, \qquad f\in \SpRn, \quad \beta \in \real .
	\end{\eq}
	It  generates the isomorphic mapping
	\begin{\eq}   \label{w4}
		W_\beta \As (\rn, w_\alpha) = \As (\rn, w_{\alpha - \beta} ). 
	\end{\eq}
    Similarly, the lift $I_\gamma$ is defined by 
	\begin{\eq}  \label{w5}
		I_\gamma: \quad f \mapsto (w_\gamma \wh{f})^\vee = (w_\gamma f^\vee)^\wedge, \qquad f\in \SpRn, \quad \gamma \in \real .
	\end{\eq}
	It  generates the isomorphic mapping
	\begin{\eq}   \label{w6}
		I_\gamma \As (\rn, w_\alpha ) = A^{s-\gamma}_{p,q} (\rn, w_\alpha) .
	\end{\eq}
	 Using these operators one can easily prove the following formula,
	\begin{\eq}   \label{w7}
		\Ft = W_{-\gamma} \circ I_{-\beta} \circ \Ft \circ W_\beta \circ I_\gamma \qquad \text{in} \quad \SpRn.
	\end{\eq}
\begin{proof}
	{\em Step 1.} 
	Let $2\le p\le \infty$ and $\frac{1}{p} + \frac{1}{p'} =1$. Then one has by the well--known embeddings according to \cite[Proposition
2.3.2/2, p.\,47, 2.5.7, p.\,89]{T83} that
\begin{\eq}   \label{w8}
	L_p (\rn) \hra B^0_p (\rn), \qquad B^0_{p'} (\rn) \hra L_{p'} (\rn),
\end{\eq}
and by the Hausdorff--Young inequality  that
\begin{\eq}    \label{w10}
	\Ft: \quad B^0_{p'} (\rn) \os{\id}{\hra}  L_{p'} (\rn) \os{\Ft}{\hra} L_p (\rn) \os{\id}{\hra} B^0_p (\rn).
\end{\eq}  
	We prove that $\Ft$ in \eqref{T3.5_1} is continuous by reduction to \eqref{w10}, based on \eqref{w4}, \eqref{w6},
	\begin{\eq}   \label{w11}
		\begin{aligned}
			I_s : \quad & B^s_{p'} (\rn, w_\sigma) && \hra B^0_{p'} (\rn, w_\sigma), \\
			W_\sigma : \quad & B^0_{p'} (\rn, w_\sigma) && \hra B^0_{p'} (\rn), \\
			\Ft : \quad & B^0_{p'} (\rn) && \hra B^0_p (\rn), \\
			I_{-\sigma}: \quad & B^0_p (\rn) && \hra B^\sigma_p (\rn), \\
			W_{-s}: \quad & B^\sigma_p (\rn) && \hra B^\sigma_p (\rn, w_s).
		\end{aligned}
	\end{\eq}
	Then \eqref{T3.5_1} follows from \eqref{w7}. The proof of \eqref{T3.5_2} follows the same scheme, relying on the Hausdorff-Young inequality,
	$H^0_{p'} (\rn) = L_{p'} (\rn)$ and again \eqref{w4}, \eqref{w6}.\\
	
	{\em Step 2.} We prove that the mappings in \eqref{T3.5_1} and \eqref{T3.5_2} are not compact. The mappings $I_s$, $I_{-\sigma}$ and
	$W_\sigma$, $W_{-s}$ in \eqref{w11} and in its $H$--counterpart are isomorphisms. In other words, one has to show that $\Ft$ in
	\begin{\eq}   \label{w12}
		\Ft: \quad B^0_{p'} (\rn) \hra B^0_p (\rn)
	\end{\eq}
	and in 
	\begin{\eq}   \label{w13}
		\Ft: \quad L_{p'} (\rn) \hra L_p (\rn), 
	\end{\eq}
    are not compact. 
    We rely on the same arguments as in Step 4 of the proof of Theorem~\ref{T3.1}. 
    Let $\psi = \psi_F$ be a smooth compactly supported father wavelet according to \eqref{2.17} and let $\psi_m (x) = \psi (x-m)$, $x\in \rn$,
	$m\in \zn$. Let
	\begin{\eq}   \label{w15}
		f_m (x) = \big(\Fti \psi_m \big)(x) = e^{imx} \Fti \psi, \qquad x\in \rn, \quad m\in \zn.
	\end{\eq}
	Then one has 
	\begin{\eq}   \label{w16}
		\big( \vp_j \wh{f_m} \big)^\vee (x) = (\vp_j \psi_m )^\vee (x) =0
	\end{\eq}
	with the possible  exception of terms with $|m| \sim 2^j$. We may assume $\vp_j \psi_m = \psi_m$. Then one has
	\begin{\eq}   \label{w17}
		\| (\vp_j \psi_m)^\vee | L_{p'} (\rn) \| = \| e^{imx} \psi^\vee | L_{p'} (\rn) \| = \| \psi^\vee |L_{p'} (\rn) \| 
	\end{\eq}
	and
	%
	\begin{\eq}   \label{2.208}
		\| f_m \, | B^0_{p'} (\rn) \| \sim \|f_m \, |L_{p'} (\rn) \| \sim 1, \qquad m\in \zn, \quad 1 \le p' \le 2.
	\end{\eq}
	On the other hand, $\{\Ft f_m = \psi_m: m \in \zn \}$ is not \red{precompact} in $B^0_p (\rn)$ and $L_p (\rn)$, $2 \le p \le \infty$. This proves
	that $\Ft$ in \eqref{T3.5_1} and \eqref{T3.5_2} is not compact. 
      \end{proof}

\section{Compact mappings}   \label{S4}
\subsection{Entropy numbers}   \label{S4.1}
The mapping $\Ft$ in \eqref{3.1} for $B$-spaces is compact if, and only if, both $s_1 >0$ and $s_2 >0$. We  referred  the reader so far
to \cite{Tri22a} where we dealt with the decay properties of related entropy numbers. We return to this topic being now in a much
better position. Instead of $L_p (\rn)$, $1<p<\infty$, as used in \cite{Tri22a}, we can now rely on their improvements 
\begin{\eq}  \label{4.1}
\SBxp{0} (\rn) = B^0_p (\rn), \ 2\le p \le \infty, \quad \text{and} \quad \TBxp{0} (\rn) = B^0_p  (\rn), \ 1\le p \le 2.
\end{\eq}
This gives the possibility to argue more directly, to \purple{ prove upper estimates in the case}  $p= \infty$ (what was not possible in \cite{Tri22a} where we used 
sophisticated duality arguments for entropy numbers), and to improve some estimates in limiting situations. \purple{We also give some lower estimates for any $p>1$.} What follows  may also
justify the new notation according to Definition~\ref{D2.3} despite the otherwise a little bit curious index shifting.

\begin{definition}  \label{D4.1}
Let $T: A \hra B$ be a linear and continuous mapping from a Banach space $A$ into a Banach space $B$. Then the entropy number $e_k 
(T)$, $k\in \nat$, is the infimum of all $\ve >0$ such that
\begin{\eq}   \label{4.2}
T (U_A) \subset \bigcup^{2^{k-1}}_{j=1} (b_j + \ve U_B ) \quad \text{for some} \quad b_1, \ldots, b_{2^{k-1}} \in B,
\end{\eq}
where $U_A = \{ a\in A: \|a \,| A \| \le 1 \}$ and $U_B = \{ b\in B: \,\|b\,| B \| \le 1 \}$.
\end{definition}

\begin{remark}   \label{R4.2}
  \red{For basic properties we refer to \cite{CS,EdE87,Koe,Pie-s} (restricted to the case of Banach spaces), and \cite{ET96} for some extensions to quasi-Banach spaces, further references and details may also be found in \cite[Section 1.10, pp.\,55--58]{T06}}. We only mention that the linear and
continuous mapping $T: A \hra B$ is compact if, and only if, $e_k (T) \to 0$ for $k \to \infty$.
\end{remark}

{ We use the weighted spaces with admissible weights as already defined in \eqref{w1}--\eqref{w3} and the lift operator $I_\alpha$ defined by \eqref{w5}. 
Please note that \eqref{w6} means in particular that 
\begin{\eq}  \label{4.7}
\| (w_\alpha \wh{f} )^\vee |A^{s-\alpha}_{p,q} (\rn) \| \sim \| f \, | A^s_{p,q} (\rn) \|,
\end{\eq}
what is also an immediate consequence  of \eqref{2.11} for the Sobolev spaces.}
If $p=2$, then
\begin{\eq}   \label{4.8}
\SBxx{s}{2} (\rn) = \SHxx{s}{2} (\rn) = H^s (\rn) \quad \text{and} \quad \TBxx{s}{2} (\rn) = \THxx{s}{2} (\rn) = H^{-s}(\rn)
\end{\eq}
are the related distinguished Hilbert spaces $H^\sigma (\rn) = H^\sigma_2 (\rn) = B^\sigma_2 (\rn)$, $\sigma \in \real$. We use the
notation as introduced in Definition~\ref{D2.3}.


\red{
For later use we recall briefly what is known about the entropy numbers $e_k (\id_\alpha)$, $k\in \nat$, of the compact embeddings between the spaces 
introduced above,
\begin{\eq}   \label{4.7*}
\id_\alpha: \quad B^{s_1}_{p_1, q_1} (\rn, w_\alpha) \hra B^{s_2}_{p_2, q_2} (\rn).
\end{\eq}
Let $s_i \in \real$, $p_i, q_i \in (0, \infty]$, $i=1,2$, and $\alpha \ge 0$. Let
\begin{\eq}   \label{4.8*}
\delta = s_1 - \frac{n}{p_1} - \big( s_2 - \frac{n}{p_2} \big) \quad \text{and} \quad \frac{1}{p^*} = \frac{1}{p_1} + \frac{\alpha}{n}.
\end{\eq}
Then $\id_\alpha$ in \eqref{4.7*} is compact if, and only if,
\begin{\eq}   \label{4.9*}
s_1 > s_2, \quad \delta >0, \quad \alpha >0, \quad p_2 > p^*.
\end{\eq}
This coincides with \cite[Proposition 6.29, p.\,281]{T06} \red{and mainly relies on \cite{HaT94} and \cite[Section~4.3.2]{ET96}}. According to \cite[Theorem 6.31, pp.\,282--283, Theorem 6.33, p.\,284]{T06}
one has the following (almost) final assertions for the related entropy numbers. In Figure~\ref{fig-2} below we indicate each space of type $\Bs$ by its parameters $(\frac1p,s)$ (neglecting the fine parameter $q$). For a fixed source space $B^{s_1}_{p_1,q_1}(\rn, w_\alpha)$ the embedding $\id_\alpha$ to any target space $B^{s_2}_{p_2,q_2}(\rn)$ in the shaded area is compact.
}

\begin{figure}[h]
\noindent\begin{minipage}{\textwidth}
  ~\hfill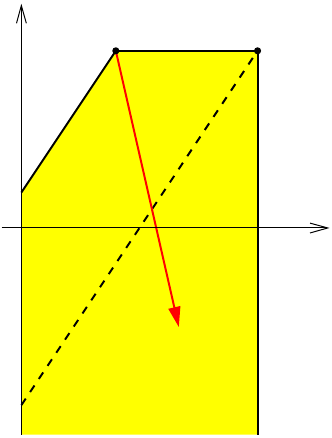\hfill~\\
~\hspace*{\fill}\unterbild{fig-2}
\end{minipage}
\end{figure}
\smallskip~
  
\red{
\begin{proposition}   \label{P4.5}
Let $e_k (\id_\alpha)$, $k\in \nat$, be the entropy numbers of the compact embedding \eqref{4.7*}. Then, as illustrated in 
Figure~\ref{fig-2}, for $k\in\nat$,
\begin{equation}   \label{4.10*}
  e_k (\id_\alpha) \sim \begin{cases}
    k^{-\frac{s_1 - s_2}{n}},& \quad \text{if} \ \delta < \alpha, \\
    k^{- \frac{\alpha}{n} + \frac{1}{p_2} - \frac{1}{p_1}}, & \quad \text{if} \ \delta > \alpha.\end{cases}
 \end{equation}
Let, in addition,
\begin{\eq}   \label{4.12*}
\tau = \frac{s_1 - s_2}{n} + \frac{1}{q_2} - \frac{1}{q_1}.
\end{\eq}
Then, for $k\in\nat$,  $k\geq 2$,
\begin{\eq}   \label{4.13*}
  e_k (\id_\alpha) \sim \begin{cases} k^{- \frac{s_1 - s_2}{n}},&  \quad \text{if} \quad \delta = \alpha, \quad \tau <0,\\
    k^{-\frac{s_1 - s_2}{n}} \big( \log k)^{\tau},&  \quad \text{if} \quad \delta = \alpha, \quad \tau >0.\end{cases}
\end{\eq}
\end{proposition}

\begin{remark}   \label{R4.6*}
  This is sufficient for our later purposes. \red{The limiting case \eqref{4.13*} was first solved in some special case in \cite{Ha97} and later solved completely in \cite{KLSS06}, see also \cite{HaT05} for some intermediate step.} 
  Further information and discussions may be found in \cite[Section 6.4, pp.\,279--286]{T06}.
\end{remark}

Now we can state our main result about the entropy numbers of the Fourier transform.
}


\begin{theorem}   \label{T4.3}
  \bli
\item
  Let $1 \le p \le \infty$, $s_1 >0$ and $s_2 >0$. Then
\begin{\eq}   \label{4.11}
\Ft: \quad \SBxp{s_1} (\rn) \hra \TBxp{s_2} (\rn) 
\end{\eq}
is compact and there exist  positive constants $C,c>0$ such that  for $k\in\nat$, $k\geq 2$, 
\begin{\eq}   \label{4.12}
e_k (\Ft) 
\le c
\begin{cases}
k^{- \frac{1}{n} \min(s_1, s_2)} &\text{if $s_1 \not= s_2$}, \\
\Big( \frac{k}{\log k} \Big)^{- \frac{s_1}{n}} &\text{if $s_1 = s_2$},
\end{cases}
\end{\eq}
and \red{for $p>1$,}
\begin{\eq}   \label{4.12a}
	e_k (\Ft) 
	\ge C
	\begin{cases}
		k^{- \frac{1}{n} \min(s_1, s_2)- |1-\frac{2}{p}|} &\text{if $s_1 \not= s_2$}, \\
		 k^{- \frac{s_1}{n}- |1-\frac{2}{p}|} \Big(\log k\Big)^{\frac{s_1}{n}}&\text{if $s_1 = s_2$},
	\end{cases}
\end{\eq}
for $k \in \nat$, $k\geq 2$. 
\item
Let $1<p<\infty$, $s_1 >0$ and $s_2 >0$. Then
\begin{\eq}   \label{4.13}
\Ft: \quad \SHxp{s_1} (\rn) \hra \THxp{s_2} (\rn)
\end{\eq}
is compact and the estimates \eqref{4.12} and \eqref{4.12a} hold. 

In particular, with $p=2$, we get for $s_1 >0$ and $s_2 >0$ that
\begin{\eq}  \label{4.9}
\Ft: \quad H^{s_1} (\rn) \hra H^{-s_2} (\rn)
\end{\eq}
is compact and
\begin{\eq}   \label{4.10}
e_k (\Ft) \sim
\begin{cases}
k^{- \frac{1}{n} \min(s_1, s_2)} &\text{if $s_1 \not= s_2$}, \\
\Big( \frac{k}{\log k} \Big)^{- \frac{s_1}{n}} &\text{if $s_1 = s_2$},
\end{cases}
\end{\eq}
$k \in \nat$, $k\geq 2$. 
\eli
\end{theorem}

\begin{proof}
{\em Step 1.}~ Note that \eqref{4.9}, \eqref{4.10} are reformulations of \cite[Theorem~4.8, part (i), pp.\,146--147]{Tri22a}. It is also a by-product of what follows.\\

{\em Step 2.}~ We first prove the upper estimates in (i) for $1\le p <2$, $s_1 >0$, $s_2 >0$.  Then one has  by \eqref{w6} the isomorphic mapping
\begin{\eq}   \label{4.14}
I_{s_1}: \quad \SBxp{s_1} (\rn) = B^{s_1 + \dnp}_p (\rn) \hra B^{\dnp}_p (\rn) = \SBxp{0} (\rn).
\end{\eq}
Combined with \eqref{3.1} one obtains
\begin{\eq}   \label{4.15}
\| w_{s_1} \wh{f} \, | B^0_p (\rn) \| \le c \, \|f \, | B^{s_1 + \dnp}_p (\rn) \|, \qquad f\in \SBxp{s_1} (\rn).
\end{\eq}
This reduces \eqref{4.11}, \eqref{4.12} to $e_k (\Ft) \le c e_k (\id)$, $k \in \nat$, of the embedding
\begin{\eq}   \label{4.16}
\id: \quad B^0_p (\rn, w_{s_1}) \hra B^{-s_2}_p (\rn), \qquad 1 \le p <2, \quad s_2 >0.
\end{\eq}
One has according to \ignore{\cite[Proposition 4.5, p.\,144]{Tri22a}, based on} \red{Proposition~\ref{P4.5}}, that
\begin{\eq}    \label{4.17}
e_k (\id) \sim
\begin{cases}
k^{- \frac{1}{n} \min(s_1, s_2)} &\text{if $s_1 \not= s_2$}, \\
\Big( \frac{k}{\log k} \Big)^{-\frac{s_1}{n}} &\text{if $s_1 = s_2$},
\end{cases}
\end{\eq}
$2 \le k \in \nat$. This proves the estimate from above in \eqref{4.12} with $1\le p \le 2$.
\\

{\em Step 3.}~ Similarly one proves \eqref{4.12} for the $B$-spaces with $2 \le p \le \infty$ (now including $p=\infty$ without
additional efforts). The $H$-counterpart in part (ii) of the above theorem can be obtained in the same way, based on Theorem~\ref{T3.3}. The crucial counterpart of \eqref{4.15} is now given by
\begin{\eq}   \label{4.18}
\| w_{s_1} \wh{f} \, | L_p (\rn) \| \le c \, \|f \, | H^{s_1 + \dnp}_p (\rn) \|, \qquad 1<p<2, \quad s_1 >0,
\end{\eq}
whereas \eqref{4.17} for 
\begin{\eq}   \label{4.19}
\id: \quad L_p (\rn, w_{s_1}) \hra H^{-s_2}_p (\rn), \qquad 1<p<\infty, \quad s_1 >0, \quad s_2 >0,
\end{\eq}
is covered by \red{Proposition~\ref{P4.5} again}, in the limiting case.\\

\emph{ Step 4.}~ Now we prove the estimates from below. First we consider a non-limiting case for Sobolev spaces. Let $2\le p<\infty$.  Then $\SHxp{s_1}(\rn)=H^{s_1}_p(\rn)$ and $\THxp{s_2}(\rn) = H^{-s_2+\dnp}_p(\rn)$, and we get the following commutative  diagram 
\begin{equation}\label{diag1} 
	\begin{tikzcd}[row sep=tiny]
		& H^{s_1}_p(\rn) \arrow[dd, "\Ft"] \\
		L_{p'}(\rn, w_{s_1}) \arrow[ur, "\Fti"] \arrow[dr, "\id"] & \\
		& H^{-s_2+\dnp}_p(\rn)
	\end{tikzcd}
\end{equation}
where we applied the Hausdorff-Young inequality \red{\eqref{T3.5_2}} to $\Fti$. Elementary properties of entropy numbers imply 
\begin{\eq}\label{new1}
	e_k\Big(\id: L_{p'}(\rn, w_{s_1}) \hookrightarrow H^{-s_2+\dnp}_p(\rn)\Big)\le C e_k(\Ft). 
\end{\eq} 
But it is known, \red{cf. Proposition~\ref{P4.5} together with the embedding $B^\sigma_{r,\min(r,2)}(\rn)\hra H^\sigma_r(\rn) \hra B^\sigma_{r,\max(r,2)}(\rn)$, $1<r<\infty$, $\sigma\in\real$, and its weighted counterpart,} that
\begin{align}\label{new2}
	e_k\Big(\id: L_{p'}(\rn, w_{s_1}) \hookrightarrow H^{-s_2+\dnp}_p(\rn)\Big) \sim  
	\begin{cases}
		k^{-\frac{s_2}{n}+\frac{2}{p}-1} &\qquad \text{if}\quad s_1>s_2,\\
		k^{-\frac{s_1}{n}+\frac{2}{p}-1} &\qquad \text{if}\quad s_1<s_2.
	\end{cases}
\end{align}
Thus \eqref{new1} and \eqref{new2} imply \red{\eqref{4.12a} for the Sobolev spaces}. In the limiting case $s_1=s_2$ we have
\begin{align}\label{new2*}
	e_k\Big(\id: L_{p'}(\rn, w_{s_1}) \hookrightarrow H^{-s_2+\dnp}_p(\rn)\Big) \ge   \red{c}\
		 k^{-\frac{s_1}{n}+\frac{2}{p}-1} (\log k)^\frac{s_1}{n},
\end{align}
and the estimate from below follows in the same way \red{by \eqref{4.13*}, since
  \begin{equation*}
    	e_k\Big(\id: L_{p'}(\rn, w_{s_1}) \hookrightarrow H^{-s_2+\dnp}_p(\rn)\Big) \ge  \red{c}\ 	e_k\Big(\id: B^0_{p'}(\rn, w_{s_1}) \hookrightarrow B^{-s_2+\dnp}_p(\rn)\Big) 
    \end{equation*}
and $\tau$ given by \eqref{4.12*} satisfies $\tau=\frac1n(s_1-\dnp) + \frac1p-\frac{1}{p'} = \frac{s_1}{n}>0$.}

In the case of the Besov spaces with $2\le p\le \infty$ we can argue similarly. Now we use the following commutative diagram 
\begin{equation}\label{diag2} 
	\begin{tikzcd}[row sep=tiny]
		& B^{s_1}_p(\rn) \arrow[dd, "\Ft"] \\
		B^0_{p'}(\rn, w_{s_1}) \arrow[ur, "\Fti"] \arrow[dr, "\id"] & \\
		& B^{-s_2+\dnp}_p(\rn)
	\end{tikzcd}
\end{equation}
where once more we benefit from the weighted version of the Hausdorff-Young inequality, \eqref{T3.5_1}. The corresponding estimate of the entropy numbers   $	e_k\Big(\id: B^0_{p'}(\rn, w_{s_1}) \hookrightarrow B^{-s_2+\dnp}_p(\rn)\Big)$ can be found \red{in Proposition~\ref{P4.5}. }

If $1 < p <2$, then we can use the duality of entropy numbers proved in \cite{AMS04} and \cite{AMST04}. This property asserts, in particular, that  if 
 one of the Banach spaces $A$ or $B$ is  uniformly convex, then for any compact operator $T:A\hookrightarrow B$ there exists a constant $c\ge 1$ and $\ell\in \nat$ such that $e_{\ell k}(T') \le c e_{ k}(T)$, where $T':B'\hookrightarrow A'$ is the dual operator to $T$, cf. \red{\cite[pp.332-333]{Pie07}}.  If $1<p<\infty$, then the spaces we deal with are uniformly convex, cf. \cite{CE88}, so the inequality \eqref{4.12a} follows by duality. 
\end{proof}


\begin{remark}  \label{R4.4}
  \red{
For $\Ft$ in \eqref{4.11} with $p=1$ we have so far only the upper estimate \eqref{4.12}. But \eqref{4.12a} remains valid also for $p=1$. For this purpose one has to modify the above interplay between Hausdorff-Young inequalities and embeddings. We do not go into detail here.
  }

Compared with \cite[Theorem~4.8, pp.\,146--147]{Tri22a} we did not only incorporate now $p= \infty$ in the $B$-part of the above 
theorem, but improved also the $\log$-perturbation in the limiting situation $s_1 = s_2$, caused there by additional embeddings (using
only \eqref{4.15} with $p=2$, because \eqref{3.1} was only available in the weaker version with $L_p (\rn)$, $1<p \le 2$, in  place of
$B^0_p (\rn)$). 
\end{remark}

\begin{remark}\label{R-ek-gap}
  It is obvious that apart from the case when $p=2$, we have a gap in the lower and upper estimate for the entropy numbers in \eqref{4.12} and \eqref{4.12a}. We do not have a complete result yet, but would like to point out, that the gap in the exponent is just characterised by $|d^n_p|$, that is, for $s_1\neq s_2$, \eqref{4.12} and \eqref{4.12a} can be summarised as
  \[
    c_1\ k^{-\frac1n \left(\min(s_1,s_2)\red{+}|d^n_p|\right)} \leq e_k(\Ft) \leq c_2\  k^{-\frac1n \min(s_1,s_2)},
    \]
and a similar expression in case of $s_1=s_2$. Though this gap spoils the optimal result so far, it may also indicate the influence of the number $|d^n_p|$, recall also Remark~\ref{R2.4}.     
\end{remark}

\subsection{Approximation numbers}   \label{S4.2}
In connection with \eqref{3.2} we relied on
\begin{\eq}   \label{4.20}
\wh{f}= \sum_{\substack{j\in \no, G\in G^j, \\ m\in \zn}} 2^{jn} (f, \wh{\psi^j_{G,m}}) \,\psi^j_{G,m}.
\end{\eq}
This representation is very well adapted to the study under which conditions the Fourier transform generates a nuclear mapping between
distinguished function spaces. This has been done in detail in \cite{HST22}. One may ask whether \eqref{4.20} can also be used to
say more about the degree of compactness measured in terms of related quantities. The first candidates are approximation numbers. But
we could not see how possible refinements may look like. Maybe the arguments used in Step~1 of the proof of Theorem~\ref{T3.1} are too
crude for this purpose. However, as an alternative one can argue as in the proof of Theorem~\ref{T4.3} with approximation numbers in 
place of entropy numbers. But first we collect some definitions and useful ingredients.

\begin{definition}   \label{D4.5}
Let $T: A \hra B$ be a linear and continuous mapping from a Banach space $A$ into a Banach space $B$. Then the approximation number
$a_k  (T)$ of $T$,
\begin{\eq}   \label{4.21}
a_k (T) = \inf \{ \|T-L \|: \ \rank L <k \}, \qquad k\in \nat,
\end{\eq}
\red{where} the infimum is taken over all linear mappings $L: A \hra B$ of $\rank L <k$, where $\rank L$ is the dimension of the range of $L$.
\end{definition}

\begin{remark}   \label{R4.6}
  Details and references may be found 
  \cite[Definition 1.87, Remark 1.88, pp.\,56--57]{T06} and \cite[Definition 2, Remark 6, 
pp.\,11--12]{ET96}, \red{we refer also to the general references given in Remark~\ref{R4.2}}. Quite obviously, $a_1 (T) = \|T \|$. Furthermore, if $\dim A \ge n$ and $\id: A \hra A$ is the identity, then
$a_k (\id) =1$ for $k=1, \ldots,n$. If $T: A \hra B$ is compact, then one has for the dual operator $T'$,
\begin{\eq}   \label{4.22}
a_k (T') = a_k (T), \qquad T': B' \hra A', \qquad k \in \nat.
\end{\eq}
Proofs of these assertions may be found in \cite[Corollary 2.4, Proposition 2.5, p.\,55]{EdE87}.
\end{remark}

Otherwise we rely on the same ingredients as in \eqref{4.7}--\eqref{4.8} and $H^\sigma (\rn) = H^\sigma_2 (\rn) = B^\sigma_2 (\rn)$,
$\sigma \in \real$. We use the same notation as in Theorem~\ref{T4.3} based on Definition~\ref{D2.3}. Let $\Bs (\rn, w_\alpha)$ be the
weighted spaces as introduced in \eqref{2.33}.

\red{
  We recall the counterpart of Proposition~\ref{P4.5} in case of approximation numbers. This goes back to \cite{Har95} and has also been quoted in \cite[Theorem~4.3.3, p.\,183]{ET96}, cf. also \cite{Ha97,Cae98,Skr05,Sk-Vyb,HaSk08}. 

  \begin{proposition}\label{prop-ak-w}
Let the above assumptions \eqref{4.8*} and \eqref{4.9*} be satisfied and $\id_\alpha$ be given by \eqref{4.7*}. Assume, in addition, that $\delta\neq \alpha$. In case of $0<p_1<2<p_2<\infty$, assume that 
$\min(\delta,\alpha)\neq\frac{n}{\min(p_1',p_2)}$ and exclude the case
$\delta>\alpha>\frac{n}{\min(p_1',p_2)}$. Then  
\begin{equation}
a_k(\id_\alpha)  \quad\sim\quad k^{-\varkappa},\quad k\in\nat,
\label{10}
\end{equation}
where $\varkappa$ is given by
\begin{equation*}
\varkappa =\left\{\!\!
\begin{array}{l@{\quad,\quad}l@{\qquad}c}
 \frac{\min(\delta,\alpha)}{n} & p_1\leq p_2\leq 2 \quad\mbox{or}\quad
2\leq p_1\leq p_2 ,& {\fbox{\bf A}} \\[2mm]
\frac{\min(\delta,\alpha)}{n} -\frac{1}{p_2}+\frac{1}{p_1} & p^\ast<p_2\leq p_1,
& \fbox{{\bf B}}\\[2mm]
\frac{s_1-s_2}{n}-\max(\frac12-\frac{1}{p_2},\frac{1}{p_1}-\frac12)  &
p_1<2<p_2 , \hfill \alpha > \delta >\frac{n}{\min(p_1',p_2)},& \fbox{\bf C}\\[2mm]
\frac{\alpha}{n} + \min(\frac12-\frac{1}{p_2},\frac{1}{p_1}-\frac12) &
p_1<2<p_2 , \hfill \delta>\alpha >\frac{n}{\min(p_1',p_2)},& \fbox{{\bf D}}\\[2mm]
\frac{\min(\delta,\alpha)}{n}\cdot \frac{\min(p_1',p_2)}{2} &
p_1<2<p_2 , \; \min(\delta,\alpha) < \frac{n}{\min(p_1',p_2)}.& \fbox{{\bf E}}
\end{array}\right.\hfill
\end{equation*}
In case of $0<p_1<2<p_2<\infty$ and $\delta>\alpha>\frac{n}{\min(p_1',p_2)}$ the estimate from below holds as given in \eqref{10}, $a_k(\id_\alpha)\geq c k^{-\varkappa}$, $k\in\nat$.
  \end{proposition}
  
}


\red{
\begin{remark}
  Note that in case of $0<p_1<2<p_2<\infty$ and $\delta>\alpha>\frac{n}{\min(p_1',p_2)}$ there is an upper estimate available in the form that for any $\varepsilon>0$ there is some $c_\varepsilon>0$ such that for all $k\in\nat$, $k\geq 2$,  $a_k(\id_\alpha) \leq c_\varepsilon k^{-\varkappa} (\log k)^{\varkappa+\varepsilon+ 1/p_2}$. But we shall not need this weaker result in our argument below.\\
In the limiting case $\delta=\alpha$ there seems to be the only sharp result in  \cite[V,\,\S 3, Thm. 9]{MynOt} which yields in case of Sobolev spaces 
\begin{equation}
a_k(\id_\alpha : H^{s_1}_{p_1}(\rn, w_\alpha)\rightarrow H^{s_2}_{p_2}) \sim
\begin{cases} k^{-\frac{\min(\delta, \alpha)}{n}}, &
\delta\neq \alpha\\[1mm] k^{-\frac{\min(\delta,\alpha)}{n}}(\log k)^{\frac{\min(\delta,\alpha)}{n}}, & \delta=\alpha,\end{cases}
\label{myn-ot}
\end{equation}
where $ s_1>s_2$, $1<p_1\leq p_2\leq 2$ {or} $2\leq p_1\leq
p_2<\infty$, $\delta>0$, $\alpha > 0$.

\red{In the usual
$\big(\frac1p,s\big)$ diagram in Figure~\ref{fig-3} we have sketched the different possible 
cases  for $\varkappa$ according to Proposition~\ref{prop-ak-w} in the case $0<p_1<2$ and
$ \,\alpha>\frac{n}{p_1'}$. The thick dashed line refers to
$\delta=\frac{n}{\min(p_1',p_2)}$. Diagrams for the other cases and further details, concerning also partial results for the limiting situations, can be found in \cite{Ha97}.}

\end{remark}}

Now we can present the counterpart of Theorem~\ref{T4.3}.

\begin{theorem}   \label{T4.7}
  \bli
  \item
    Let $s_1 >0$ and $s_2 >0$. 
    Then
\begin{\eq} \label{4.23}
\Ft: \quad H^{s_1} (\rn) \hra H^{-s_2} (\rn)
\end{\eq}
is compact and
\begin{\eq}   \label{4.24}
a_k (\Ft) \sim 
\begin{cases}
k^{- \frac{1}{n} \min(s_1, s_2)} &\text{if} \ s_1 \not= s_2, \\
\Big( \frac{k}{\log k} \Big)^{- \frac{s_1}{n}} &\text{if}\ s_1 = s_2,
\end{cases}
\end{\eq}
for $k\in\nat$, $k\geq 2$.
\item Let $1< p < \infty$, $s_1 >0$, $s_2 >0$, $s_1 \not= s_2$. Then
\begin{\eq}   \label{4.25}
\Ft: \quad \SBxp{s_1} (\rn) \hra \TBxp{s_2} (\rn)
\end{\eq}
is compact and for some constants $c_2>c_1>0$ and all $k\in\nat$ we obtain
\begin{align}   \label{4.26}
  a_k (\Ft) &\le c_2\,  k^{-\frac{1}{n} \min(s_1, s_2)},
  \intertext{and}
   a_k(\Ft) & \geq c_1 \begin{cases} 
     k^{-\frac{\min(s_1,s_2)}{n} \red{-|\frac12-\frac1p|}} 
       & \text{if } \ \frac{\min(s_1,s_2)}{n}\geq \frac{1}{\max(p,p')}, \\ k^{-\frac{\min(s_1,s_2)}{n} \frac{\max(p,p')}{2}} & \text{if } \ \frac{\min(s_1,s_2)}{n}\leq\frac{1}{\max(p,p')}.
  \end{cases}\label{4.26a}
\end{align}
In particular, if $p=2$ and $s_1\neq s_2$, then
\begin{equation}
 \label{4.26b}
a_k (\Ft : \SBxx{s_1}{2}(\rn)\hra \TBxx{s_2}{2}(\rn)) \sim k^{-\frac{1}{n} \min(s_1, s_2)},\quad k\in\nat.
\end{equation}
\item   Let $1< p <\infty$, $s_1 >0$, $s_2 >0$ and $s_1 \not= s_2$. Then
\begin{\eq}   \label{4.27}
\Ft: \quad \SHxp{s_1} (\rn) \hra \THxp{s_2} (\rn)
\end{\eq}
is compact and $a_k (\Ft)$ can be estimated from above and below as in \eqref{4.26} and \eqref{4.26a}.
\eli
\end{theorem}

\begin{figure}[h]
    ~\hfill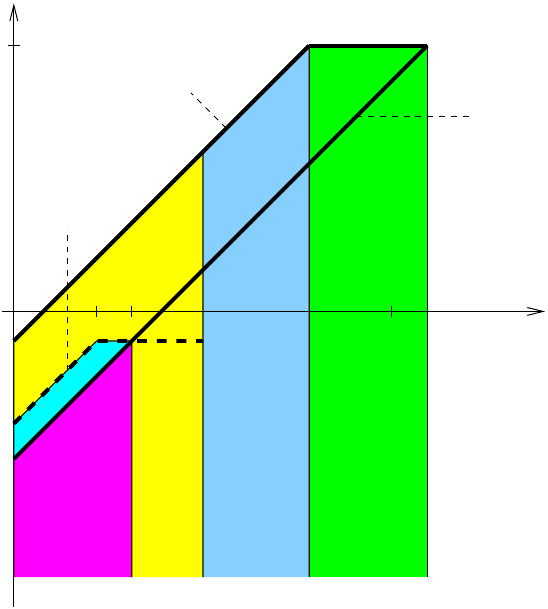\hfill~\\
    ~\hspace*{\fill}\unterbild{fig-3}
\end{figure}

\begin{remark}
Obviously, if $p=p'=2$, then \eqref{4.26} together with \eqref{4.26a} coincide with \eqref{4.24}, recall ~\eqref{2.16a}, always assuming $s_1\neq s_2$, while for $p\neq 2$ we have a gap between the upper and lower estimate in (ii) and (iii).
  \end{remark}

\begin{proof} {\em Step~1}.~ We first deal with the estimates from above. Let $0<p<\infty$, $\sigma_1 >\sigma_2$, $\alpha >0$ and $0< q_1, q_2 \le \infty$. Then
\begin{\eq}   \label{4.28}
\id: B^{\sigma_1}_{p, q_1} (\rn, w_\alpha) \hra B^{\sigma_2}_{p, q_2} (\rn)
\end{\eq}
is compact and
\begin{\eq}   \label{4.29}
a_k (\id) \sim k^{-\frac{1}{n} \min(\sigma_1 - \sigma_2, \alpha)}, \qquad \sigma_1 - \sigma_2 \not= \alpha,
\end{\eq}
$k\in \nat$. This refers \red{to the case \fbox{\bf A} in Proposition~\ref{prop-ak-w}}. Specified to the above
situation one has for
\begin{\eq}   \label{4.30}
\id: \quad B^0_{p, q_1} (\rn, w_{s_1}) \hra B^{-s_2}_{p, q_2} (\rn), \qquad 1 \le p \le 2,
\end{\eq}
$s_1$, $s_2 >0$ with $s_1 \not= s_2$, that
\begin{\eq}   \label{4.31}
a_k (\id) \sim k^{- \frac{1}{n} \min(s_1, s_2)}, \qquad k \in \nat.
\end{\eq}
In case of $p=2$ the result \eqref{myn-ot} gives
for  Sobolev spaces in the limiting situation $s_1=s_2$ that
\begin{\eq}   \label{4.31a}
a_k (\id:H^0(\rn, w_{s_1})\hra H^{-s_1}(\rn)) \sim \Big( \frac{k}{\log k} \Big)^{- \frac{s_1}{n}}, \qquad k \in \nat, \ k\geq 2,
\end{\eq}
see also \cite{Har95,Ha97}. Now one can argue in the same way as in the proof of Theorem~\ref{T4.3}, using the multiplicativity of approximation numbers, $a_k(\Ft) \leq c\ a_k(\id)$, $k\in\nat$. This covers immediately the estimates from above in part (i), part (ii) with $1\le p <2$ and part (iii) with $1<p<2$. The estimate from above in the remaining  cases for $p>2$ in (ii) and (iii) are a matter of $\Ft' =\Ft$, the duality \eqref{4.22}, \eqref{3.20} (including now $p' =\infty$) and
its $H$-counterpart, similarly as in \eqref{3.23}. \\

{\em Step~2}.~Now we study the estimates from below. In case of (i) we benefit from the isomorphisms $\Ft, \Fti: L_2(\rn, w_{s_1})\hra H^{s_1}(\rn)$, and decompose
\begin{equation}
\begin{tikzcd}[row sep=tiny]
& H^{s_1}(\rn) \arrow[dd, "\Ft"] \\
L_2(\rn, w_{s_1}) \arrow[ur, "\Fti"] \arrow[dr, "\id"] & \\
& H^{-s_2}(\rn)
\end{tikzcd}
\end{equation}
  Thus, using again the multiplicativity of approximation numbers, we arrive at 
  \[ a_k(\Ft : H^{s_1}(\rn)\hra H^{-s_2}(\rn)) \geq c\ a_k(\id: L_2(\rn, w_{s_1})\hra H^{-s_2}(\rn)), 
    \]
    and complete the argument for \eqref{4.24} in view of \eqref{4.31}, \eqref{4.31a}.\\

    Now we deal with (iii) and assume first $2<p<\infty$. Then $\SHxp{s_1}(\rn)=H^{s_1}_p(\rn)$ and $\THxp{s_2}(\rn) = H^{-s_2+\dnp}_p(\rn)$, and 
    we can used once more the commutative diagram \eqref{diag1}. 
In the same way as above we can thus conclude that
\begin{equation}\label{ak-below}
  a_k(\Ft: \SHxp{s_1}(\rn) \hra \THxp{s_2}(\rn) ) \geq c\ a_k(\id: L_{p'}(\rn, w_{s_1}) \hra H^{-s_2+\dnp}_p(\rn)).
\end{equation}
For the latter embedding we \red{apply Proposition~\ref{prop-ak-w}}, that is,
\begin{equation*}
  a_k(\id: L_{p'}(\rn, w_{s_1}) \hra H^{-s_2+\dnp}_p(\rn)) \geq c\ \begin{cases} 
k^{-\frac{\min(s_1,s_2)}{n} - \frac12 + \frac1p} & \text{if}\quad \frac{\min(s_1,s_2)}{n}\geq \frac{1}{p}, \\ k^{-\frac{\min(s_1,s_2)}{n} \frac{p}{2}} & \text{if}\quad \frac{\min(s_1,s_2)}{n}\leq\frac{1}{p},
  \end{cases}
  \end{equation*}
  which together with \eqref{ak-below} leads to
  \begin{equation*}
a_k(\Ft: \SHxp{s_1}(\rn) \hra \THxp{s_2}(\rn) ) \geq c\ \begin{cases} 
k^{-\frac{\min(s_1,s_2)}{n} - \frac12 + \frac1p} & \text{if}\quad \frac{\min(s_1,s_2)}{n}\geq \frac{1}{p}, \\ k^{-\frac{\min(s_1,s_2)}{n} \frac{p}{2}} & \text{if}\quad \frac{\min(s_1,s_2)}{n}\leq\frac{1}{p}.\end{cases}
  \end{equation*}
  If $1<p<2$, then by the duality \eqref{4.22} and the preceding result we arrive at 
  \begin{equation*}
a_k(\Ft: \SHxp{s_1}(\rn) \hra \THxp{s_2}(\rn) ) \geq c\ \begin{cases} 
k^{-\frac{\min(s_1,s_2)}{n} + \frac12 - \frac1p} & \text{if}\quad \frac{\min(s_1,s_2)}{n}\geq \frac{1}{p'}, \\ k^{-\frac{\min(s_1,s_2)}{n} \frac{p'}{2}} & \text{if}\quad \frac{\min(s_1,s_2)}{n}\leq\frac{1}{p'}.
  \end{cases}
  \end{equation*}
Both cases can be summarised to
    \begin{align}\nonumber
      a_k(\Ft: \ & \SHxp{s_1}(\rn) \hra \THxp{s_2}(\rn) ) \\
      & \geq c\ \begin{cases} 
k^{-\frac{\min(s_1,s_2)}{n} + \min(\frac12 - \frac1p,\frac1p-\frac12)} & \text{if}\quad \frac{\min(s_1,s_2)}{n}\geq \frac{1}{\max(p,p')}, \\ k^{-\frac{\min(s_1,s_2)}{n} \frac{\max(p,p')}{2}} & \text{if}\quad \frac{\min(s_1,s_2)}{n}\leq\frac{1}{\max(p,p')}.
  \end{cases}\label{a_k-pnot2}
    \end{align}
The argument in the situation $\Ft: \SBxp{s_1}(\rn) \hra   \TBxp{s_2}(\rn)$ studied in (ii) works completely parallel, \red{recall Proposition~\ref{prop-ak-w}}. Now we use the diagram \eqref{diag2} instead of \eqref{diag1}.      
\end{proof}


\begin{remark}   \label{R4.8}
  Note that in contrast to Theorem~\ref{T3.1} we excluded in Theorem~\ref{T4.7}(ii) the limiting cases $p\in \{1,\infty\}$, mainly due to our proof method via Hausdorff-Young inequality and application of duality arguments. Apart from this, we have a complete and sharp result in case of $p=2$ only. 
 But we are not aware of any paper where this gap has been sealed afterwards.
\end{remark}






\bigskip\bigskip~

{\small
\noindent
Dorothee D. Haroske\\
Institute of Mathematics \\
Friedrich Schiller University Jena\\
07737 Jena\\
Germany\\
{\tt dorothee.haroske@uni-jena.de}\\[4ex]
%
Leszek Skrzypczak\\
Faculty of Mathematics \& Computer Science\\
Adam  Mickiewicz University\\
ul. Uniwersytetu Pozna\'nskiego 4\\
61-614 Pozna\'n\\
Poland\\
{\tt lskrzyp@amu.edu.pl}\\[4ex]
Hans Triebel\\
Institute of Mathematics \\
Friedrich Schiller University Jena\\
07737 Jena\\
Germany\\
{\tt hans.triebel@uni-jena.de}
}

\end{document}

%% file: FT-rev_1.pdf_t
\begin{picture}(0,0)%
\includegraphics{FT-rev_1.pdf}%
\end{picture}%
\setlength{\unitlength}{4144sp}%
\begingroup\makeatletter\ifx\SetFigFont\undefined%
\gdef\SetFigFont#1#2#3#4#5{%
  \reset@font\fontsize{#1}{#2pt}%
  \fontfamily{#3}\fontseries{#4}\fontshape{#5}%
  \selectfont}%
\fi\endgroup%
\begin{picture}(2589,5199)(349,-4573)
\put(856,-3436){\makebox(0,0)[lb]{\smash{{\SetFigFont{10}{12.0}{\familydefault}{\mddefault}{\updefault}{\color[rgb]{0,0,0}$s=\tnp$}%
}}}}
\put(2386,-376){\makebox(0,0)[lb]{\smash{{\SetFigFont{10}{12.0}{\familydefault}{\mddefault}{\updefault}{\color[rgb]{0,0,0}$s=\dnp$}%
}}}}
\put(1846,-1726){\makebox(0,0)[lb]{\smash{{\SetFigFont{10}{12.0}{\familydefault}{\mddefault}{\updefault}{\color[rgb]{0,0,0}$\Ft$}%
}}}}
\put(1621,-1726){\makebox(0,0)[rb]{\smash{{\SetFigFont{10}{12.0}{\familydefault}{\mddefault}{\updefault}{\color[rgb]{0,0,0}$s=\snp$}%
}}}}
\put(631,-2221){\makebox(0,0)[b]{\smash{{\SetFigFont{10}{12.0}{\familydefault}{\mddefault}{\updefault}$L_p$}}}}
\put(1261,-2176){\makebox(0,0)[b]{\smash{{\SetFigFont{10}{12.0}{\familydefault}{\mddefault}{\updefault}$\frac12$}}}}
\put(2251,-2176){\makebox(0,0)[b]{\smash{{\SetFigFont{10}{12.0}{\familydefault}{\mddefault}{\updefault}$1$}}}}
\put(1666, 29){\makebox(0,0)[b]{\smash{{\SetFigFont{10}{12.0}{\familydefault}{\mddefault}{\updefault}$B^{s_1+\snp}_p=\SBxp{s_1}$}}}}
\put(1666,-4156){\makebox(0,0)[b]{\smash{{\SetFigFont{10}{12.0}{\familydefault}{\mddefault}{\updefault}$B^{\tnp-s_2}_p=\TBxp{s_2}$}}}}
\put(406,-4201){\makebox(0,0)[rb]{\smash{{\SetFigFont{10}{12.0}{\familydefault}{\mddefault}{\updefault}$-n$}}}}
\put(2701,-2491){\makebox(0,0)[lb]{\smash{{\SetFigFont{10}{12.0}{\familydefault}{\mddefault}{\updefault}$\frac1p$}}}}
\put(406,389){\makebox(0,0)[rb]{\smash{{\SetFigFont{10}{12.0}{\familydefault}{\mddefault}{\updefault}$s$}}}}
\put(406,-3931){\makebox(0,0)[rb]{\smash{{\SetFigFont{10}{12.0}{\familydefault}{\mddefault}{\updefault}$-s_2$}}}}
\put(406,-1096){\makebox(0,0)[rb]{\smash{{\SetFigFont{10}{12.0}{\familydefault}{\mddefault}{\updefault}$s_1$}}}}
\put(406,-106){\makebox(0,0)[rb]{\smash{{\SetFigFont{10}{12.0}{\familydefault}{\mddefault}{\updefault}$s_1+\snp$}}}}
\put(1756,-2176){\makebox(0,0)[rb]{\smash{{\SetFigFont{10}{12.0}{\familydefault}{\mddefault}{\updefault}$\frac1p$}}}}
\end{picture}%

%% file: FT-rev_2.pdf_t
\begin{picture}(0,0)%
\includegraphics{FT-rev_2.pdf}%
\end{picture}%
\setlength{\unitlength}{4144sp}%
\begingroup\makeatletter\ifx\SetFigFont\undefined%
\gdef\SetFigFont#1#2#3#4#5{%
  \reset@font\fontsize{#1}{#2pt}%
  \fontfamily{#3}\fontseries{#4}\fontshape{#5}%
  \selectfont}%
\fi\endgroup%
\begin{picture}(2524,3319)(289,-3008)
\put(1651,-2166){\makebox(0,0)[b]{\smash{{\SetFigFont{10}{12.0}{\familydefault}{\mddefault}{\updefault}$\bullet$}}}}
\put(1306,-511){\makebox(0,0)[lb]{\smash{{\SetFigFont{10}{12.0}{\familydefault}{\mddefault}{\updefault}{\color[rgb]{0,0,0}$\id_\alpha$}%
}}}}
\put(2551,-1591){\makebox(0,0)[lb]{\smash{{\SetFigFont{10}{12.0}{\familydefault}{\mddefault}{\updefault}$\frac1p$}}}}
\put(1651,-2301){\makebox(0,0)[b]{\smash{{\SetFigFont{10}{12.0}{\familydefault}{\mddefault}{\updefault}$\big(\frac{1}{p_2},s_2\big)$}}}}
\put(1576,-1186){\makebox(0,0)[lb]{\smash{{\SetFigFont{10}{12.0}{\familydefault}{\mddefault}{\updefault}$\delta=\alpha$}}}}
\put(2251, 89){\makebox(0,0)[b]{\smash{{\SetFigFont{10}{12.0}{\familydefault}{\mddefault}{\updefault}$\big(\frac{1}{p^*}, s_1\big)$}}}}
\put(766,-286){\makebox(0,0)[b]{\smash{{\SetFigFont{10}{12.0}{\familydefault}{\mddefault}{\updefault}$\delta=0$}}}}
\put(1171, 89){\makebox(0,0)[b]{\smash{{\SetFigFont{10}{12.0}{\familydefault}{\mddefault}{\updefault}$\big(\frac{1}{p_1},s_1\big)$}}}}
\put(401, 29){\makebox(0,0)[rb]{\smash{{\SetFigFont{10}{12.0}{\familydefault}{\mddefault}{\updefault}$s$}}}}
\end{picture}%

%% file: FT-rev_3.pdf_t
\begin{picture}(0,0)%
\includegraphics{FT-rev_3.pdf}%
\end{picture}%
\setlength{\unitlength}{4144sp}%
\begingroup\makeatletter\ifx\SetFigFont\undefined%
\gdef\SetFigFont#1#2#3#4#5{%
  \reset@font\fontsize{#1}{#2pt}%
  \fontfamily{#3}\fontseries{#4}\fontshape{#5}%
  \selectfont}%
\fi\endgroup%
\begin{picture}(4167,4624)(346,-3898)
\put(4276,-1861){\makebox(0,0)[lb]{\smash{{\SetFigFont{10}{12.0}{\familydefault}{\mddefault}{\updefault}$\frac1p$}}}}
\put(1936,-1861){\makebox(0,0)[lb]{\smash{{\SetFigFont{10}{12.0}{\familydefault}{\mddefault}{\updefault}$\frac12$}}}}
\put(2701,539){\makebox(0,0)[b]{\smash{{\SetFigFont{10}{12.0}{\familydefault}{\mddefault}{\updefault}$\big(\frac{1}{p_1},s_1\big)$}}}}
\put(3601,539){\makebox(0,0)[b]{\smash{{\SetFigFont{10}{12.0}{\familydefault}{\mddefault}{\updefault}$\big(\frac{1}{p^\ast},s_1\big)$}}}}
\put(901,-3166){\makebox(0,0)[b]{\smash{{\SetFigFont{10}{12.0}{\familydefault}{\mddefault}{\updefault}{\bf D}}}}}
\put(2296,-561){\makebox(0,0)[b]{\smash{{\SetFigFont{10}{12.0}{\familydefault}{\mddefault}{\updefault}{\bf A}}}}}
\put(1501,-1161){\makebox(0,0)[b]{\smash{{\SetFigFont{10}{12.0}{\familydefault}{\mddefault}{\updefault}{\bf E}}}}}
\put(2751,-1861){\makebox(0,0)[lb]{\smash{{\SetFigFont{10}{12.0}{\familydefault}{\mddefault}{\updefault}$\frac{1}{p_1}$}}}}
\put(3701,-1861){\makebox(0,0)[lb]{\smash{{\SetFigFont{10}{12.0}{\familydefault}{\mddefault}{\updefault}$\frac{1}{p^\ast}$}}}}
\put(1756, 74){\makebox(0,0)[rb]{\smash{{\SetFigFont{10}{12.0}{\familydefault}{\mddefault}{\updefault}$\delta=0$}}}}
\put(901,-871){\makebox(0,0)[b]{\smash{{\SetFigFont{10}{12.0}{\familydefault}{\mddefault}{\updefault}$\delta=\frac{n}{p_1'}$}}}}
\put(3151,-1096){\makebox(0,0)[b]{\smash{{\SetFigFont{10}{12.0}{\familydefault}{\mddefault}{\updefault}{\bf B}}}}}
\put(1101,-1461){\makebox(0,0)[b]{\smash{{\SetFigFont{10}{12.0}{\familydefault}{\mddefault}{\updefault}$\frac{1}{p_1'}$}}}}
\put(381,-2491){\makebox(0,0)[rb]{\smash{{\SetFigFont{10}{12.0}{\familydefault}{\mddefault}{\updefault}$s_1-n$}}}}
\put(4046,-161){\makebox(0,0)[lb]{\smash{{\SetFigFont{10}{12.0}{\familydefault}{\mddefault}{\updefault}$\delta=\alpha$}}}}
\put(381,-1861){\makebox(0,0)[rb]{\smash{{\SetFigFont{10}{12.0}{\familydefault}{\mddefault}{\updefault}$s_1-\frac{n}{p_1}$}}}}
\put(381,389){\makebox(0,0)[rb]{\smash{{\SetFigFont{10}{12.0}{\familydefault}{\mddefault}{\updefault}$s_1$}}}}
\put(381,-2761){\makebox(0,0)[rb]{\smash{{\SetFigFont{10}{12.0}{\familydefault}{\mddefault}{\updefault}$s_1-\frac{n}{p_1}-\alpha$}}}}
\put(361,564){\makebox(0,0)[rb]{\smash{{\SetFigFont{10}{12.0}{\familydefault}{\mddefault}{\updefault}$s$}}}}
\put(3381,-1861){\makebox(0,0)[b]{\smash{{\SetFigFont{10}{12.0}{\familydefault}{\mddefault}{\updefault}$1$}}}}
\put(1371,-1461){\makebox(0,0)[b]{\smash{{\SetFigFont{10}{12.0}{\familydefault}{\mddefault}{\updefault}$\frac{\alpha}{n}$}}}}
\put(811,-2301){\makebox(0,0)[b]{\smash{{\SetFigFont{10}{12.0}{\familydefault}{\mddefault}{\updefault}{\bf C}}}}}
\end{picture}%